\documentclass[11pt,reqno,a4paper]{amsart}
\usepackage{amsmath,amscd,amstext,amsthm,amsfonts,latexsym}
\usepackage{graphicx, graphics}
\usepackage{import}
\usepackage{textcomp}
\usepackage{tikz}
\usetikzlibrary{decorations.fractals}
\usepackage{a4wide}
\usepackage{cite}
\usepackage{amsmath,amssymb,amsthm,mathrsfs}
\usepackage{color}
\usepackage{mathtools}

\usepackage[bookmarks=false]{hyperref}

\DeclarePairedDelimiter\floor{\lfloor}{\rfloor}

\usepackage{enumerate}
\usepackage{enumitem}
\makeatletter \@addtoreset{equation}{section} \makeatother

\renewcommand\thetable{\thesection.\@arabic\c@table}

\theoremstyle{plain}
\newtheorem{theorem}{Theorem}[section]
\newtheorem{proposition}[theorem]{Proposition}
\newtheorem{lemma}[theorem]{Lemma}

\theoremstyle{definition}
\newtheorem{remark}[theorem]{Remark}

\newtheorem{definition}[theorem]{Definition}

\newtheorem*{condition}{Condition}

\newcommand{\field}[1]{\mathbb{#1}}

\newcommand{\NN}{\field{N}}

\usepackage[T2A,T1]{fontenc}
\DeclareSymbolFont{cyrillic}{T2A}{cmr}{m}{n}
\DeclareMathSymbol{\D}{\mathalpha}{cyrillic}{196}

\newcommand{\ro}{\rho}

\newcommand{\N}{\mathbb{N}}

\newcommand{\R}{\mathbb{R}}

\def\I{\ensuremath{{\bf 1}}}

\newcommand\phantomarrow[2]{%
  \setbox0=\hbox{$\displaystyle #1\to$}%
  \hbox to \wd0{%
    $#2\mapstochar
     \cleaders\hbox{$\mkern-1mu\relbar\mkern-3mu$}\hfill
     \mkern-7mu\rightarrow$}%
  \,}

\def\R{\ensuremath{\mathbb R}}
\def\CC{\ensuremath{\mathscr C}}
\def\N{\ensuremath{\mathbb N}}

\def\I{\ensuremath{{\bf 1}}}
\def\e{{\ensuremath{\rm e}}}

\def\B{\ensuremath{\mathcal B}}
\def\M{\ensuremath{\mathcal M}}
\def\C{\ensuremath{\mathcal C}}

\def\AA{\ensuremath{\mathcal A}}

\def\X{\mathcal{X}}
\def\M{\mathcal{M}}

\def\ie{{\em i.e.}, }

\def\cv{\ensuremath{\text {Cor}}}
\newcommand{\dif}{\mathrm{d}}

\author[A. C. M. Freitas]{Ana Cristina Moreira Freitas}
\address{Ana Cristina Moreira Freitas\\ Centro de Matem\'{a}tica \&
Faculdade de Economia da Universidade do Porto\\ Rua Dr. Roberto Frias \\
4200-464 Porto\\ Portugal} \email{\href{mailto:amoreira@fep.up.pt}{amoreira@fep.up.pt}}
\urladdr{\url{http://www.fep.up.pt/docentes/amoreira/}}

\author[J. M. Freitas]{Jorge Milhazes Freitas}
\address{Jorge Milhazes Freitas\\ Centro de Matem\'{a}tica \& Faculdade de Ci\^encias da Universidade do Porto\\ Rua do
Campo Alegre 687\\ 4169-007 Porto\\ Portugal}
\email{\href{mailto:jmfreita@fc.up.pt}{jmfreita@fc.up.pt}}
\urladdr{\url{http://www.fc.up.pt/pessoas/jmfreita/}}

\author[F.B. Rodrigues]{Fagner B. Rodrigues}
\address{Fagner Bernardini Rodrigues\\Instituto de Matem\'atica - Universidade Federal do Rio Grande do Sul\\
Av. Bento Gonçalves, 9500 - Pr\'edio 43-111 - Agronomia\\
Caixa Postal 15080\
91509-900 Porto Alegre - RS - Brasil
} \email{fagnerbernardini@gmail.com }

\author[J.V. Soares]{Jorge Valentim Soares}
\address{Jorge Valentim Soares\\ Centro de Matem\'{a}tica \& Faculdade de Ci\^encias da Universidade do Porto\\ Rua do
Campo Alegre 687\\ 4169-007 Porto\\ Portugal}
\email{\href{mailto:jsoares235@gmail.com}{jsoares235@gmail.com}}

\thanks{All authors were partially supported by FCT projects FAPESP/19805/2014, PTDC/MAT-CAL/3884/2014 and PTDC/MAT-PUR/28177/2017, with national funds, and by CMUP (UID/MAT/00144/2019), which is funded by FCT with national (MCTES) and European structural funds through the programs FEDER, under the partnership agreement PT2020. We thank Vuksan Mijovi\'{c}, Mike Todd and Romain Aimino for helpful comments and suggestions.}


\begin{document}

\title{Rare events for Cantor target sets }
\date{\today}

\maketitle
\begin{abstract}
We study the existence of limiting laws of rare events corresponding to the entrance of the orbits on certain target sets in the phase space. The limiting laws are obtained when the target sets shrink to a Cantor set of zero Lebesgue measure. We consider both the presence and absence of clustering, which is detected by the Extremal Index, which turns out to be very useful to identify the compatibility between the dynamics and the fractal structure of the limiting Cantor set.  
The computation of the Extremal Index is connected to the box dimension of the intersection between the Cantor set and its iterates.

\end{abstract}
\section{Introduction}

The study of rare events for dynamical systems has experienced a vast development in the last two decades (see the book \cite{LFFF16} and the review paper \cite{S09}) and motivated, in particular, applications to climate dynamics (see for example \cite{SKFG16,MCF17,MCBF18,CFMV18}). The occurrence of rare events is tied to the entrance of the orbit in a sensitive region of the phase space, with small measure, which justifies the use of the word rare. There are two main approaches to the subject: one is through the study of the distribution of the normalised elapsed time that the orbits take to hit or return to  such regions of the phase space, which we will refer to the study of Hitting/Return Times Statistics (HTS/RTS), and the other is through the study of the extremal behaviour (distribution of the maximum) of stochastic processes arising from the system simply by evaluating a given observable $\varphi$ through the orbits of the system.

The two approaches were proved to be equivalent \cite{C01,FFT10,FFT11}, when the points where the observable function $\varphi$ exceeds a high threshold correspond exactly to the sensitive region, which is used for target set for the study of HTS/RTS. Then the underlying idea is that the stochastic process showing no exceedances of a certain high threshold up to time $n$ means that the hitting/return time to the respective target set must be larger than $n$. The limiting laws are obtained when the threshold increases to its maximum value, which means that the target sets shrink to the maximal set, $\mathcal M$,  where the observable function $\varphi$ achieves its global maximum value. 

In the existing literature regarding the study of rare events for dynamical systems (in both approaches), most of the times, the set $\mathcal M$ is reduced to a single point. However, in a few papers, $\mathcal M$ has been chosen to be a finite set of points (\cite{HNT12,AFFR16}), a countable set (\cite{AFFR17}), a one dimensional submanifold, such as the diagonal of product spaces (\cite{CCC09,KL09,FGGV18}), but has always been taken with a regular geometric structure.
The exception is the paper \cite{MP16}, which served as motivation for the present work, where the authors consider the situation of fractal landscapes, with $\mathcal M$ taken as a Cantor set. They conjectured that the same distributional limits observed when $\mathcal M$ was a singular point should apply for such more intricate maximal sets. Their contribution can be described, in their own words, as experimental mathematics: while the definitions of the objects under examination, as well as the form of the conjectured limiting laws, were complete and rigorous and, most of the times, the value of various constants involved were obtained from explicit computations, they did not provide proofs of the conjectured results, which were supported by numerical simulation studies. Their study also revealed the importance played by the Minkowski dimension and Minkowski content in the choice of the normalising sequences in order to recover the classical distributional limits.

The motivation to use maximal sets with a finer geometrical structure comes from the possibility of applications to real life situations when one has many variables and the sensitive regions of the phase space are described as fractal landscapes. Cases such as mine swiping, the movement of air masses, road traffic, network communications, structural safety, stock market, where one is particularly worried with the occurrence of certain critical configurations that correspond to sensitive regions with a complex structure, are very common in the real world. Therefore, understanding the extremal dynamics of simpler lower dimensional models, but which still capture the landscape fractal complexity of the critical regions, is of the utmost importance.

Hence, in this work we assume that $\M$ is a Cantor set and prove that the conjectured limit behaviour observed in \cite{MP16} holds true for uniformly expanding dynamics. To our knowledge, this is the first time that rare events limiting laws are proved analytically for a limiting target set (or maximal set) $\mathcal M$ with a fractal geometry. Moreover, we study the possibility of occurring clustering of extreme events for such maximal sets. We remark that, in all the examples  considered in \cite{MP16}, with fractal maximal sets with Hausdorff dimension strictly larger than 0, clustering was not detected. Here, not only do we provide examples which show that clustering can still occur in such situations, as we explain the mechanism responsible for the appearance of clustering, by using tools from fractal geometry. 

As shown in \cite{FFT12,AFFR16}, clustering of rare events is connected with the recurrence properties of $\mathcal M$ by the system's dynamics. Of course that, when $\mathcal M$ is reduced to a single point, then clustering is related to the periodicity of that point. In fact, in \cite{FFT12,K12,FP12,AFV15,FFTV16}, a dichotomy was proved for uniformly and non-uniformly hyperbolic systems: either the single point of $\mathcal M$ is periodic and we have clustering, or is non-periodic and we have the absence of clustering. When $\mathcal M$ has finitely many or countably many points, as shown in \cite{AFFR16,AFFR17}, then either the orbits of those points collide with $\mathcal M$ creating  clustering or they do not hit $\M$ and in that case, as observed also in \cite{HNT12}, there is no clustering. This explains the detection of clustering in the last example of \cite{MP16}, where $\mathcal M$ was a countable set of points.

The presence of clustering is detected by the Extremal Index (EI), which is a parameter that takes values between 0 and 1 and appears as an exponent in the usual exponential limiting law. When there is no clustering the EI is equal to 1, while the presence of clustering leads to an EI less than 1. The more intense is the clustering the smaller is the EI.

As seen in \cite{FFT12,AFFR16}, for a nice discrete time system $T:\X\to\X$, when $\M$ is finite then an EI less than 1 implies that the orbits of the points of $\M$ hit the maximal set itself, \ie there exists some $q\in\N$ such that $T^{-q}(\M)\cap\M\neq\emptyset$. When dealing with an infinite but countable extremal set $\M$, typically, a fast recurrence from $\M$ to itself produces an EI less than 1, as in the finite case. Nevertheless, in some special cases like \cite[Example 4.7]{AFFR17}, a very slow recurrence of the orbits of the maximal points to $\M$ may turn the clustering negligible so that in the limit the EI is still 1. We will see that when $\M$ is uncountable the situation is more complex and the value of the EI is linked to the finer geometrical properties of $\M$, namely, to its fractal dimension and thickness.

We will see (Sections~\ref{sec:statement-of-results} and \ref{sec:EI-indicator}) that one may have a large and fast recurrence of the orbits of the points of $\M$ to itself, \ie the set $T^{-q}(\M)\cap\M$ may be even infinite for small $q$, and yet the EI is still 1. In the context of Theorems~\ref{thm:2x_mod1} and \ref{thm:3kx_mod1}, the EI will only be less than 1 if one of the intersections $T^{-q}(\M)\cap\M$, for $q\in\N$, is relevant in the sense that its box dimension is equal to the box dimension of $\M$, otherwise we will always get an EI equal to 1. Hence, the EI indicates how the dynamics of $T$ is or is not compatible with the geometric structure of $\M$.

We will illustrate this behaviour analytically with simple models and dynamics, namely, in Section~\ref{sec:statement-of-results}, the maximal set $\M$ will be the ternary Cantor set and $T$ will be a uniformly expanding map of the form $mx\mod 1$, with $m\in\N$. This simplification allows to compute concrete estimates for the box dimension of the intersections mentioned earlier and exact values for the EI. In the case of presence of clustering created by a clear compatibility between the dynamics and the self-similarity structure of the maximal set, we will consider more general Cantor sets (see Section~\ref{se:3x_mod1}). We remark that although we work with simple models, they capture the essence of the limiting behaviour of the statistics of rare events dynamics and, in fact, we believe that the spirit of our findings should prevail in more irregular situations, as the numerical simulation study performed in Section~ \ref{sec:EI-indicator} suggests.

As seen in \cite{LFTV12}, the authors considered $\M$ reduced to a single point chosen in the support of a dynamical attractor and have shown that extremes can be thought of as geometric indicators of the local properties of the attractor. Here, $\M$ is an intricate, much more general set, which means that the extremes, instead of a local information provide a global one. In fact, we believe that an interesting byproduct of our results is that the EI could be used as an indicator of the compatibility of a certain dynamics with the geometric self-similarity structure of $\M$. Namely, in the cases considered, we obtained that the EI is always 1, except for the cases when $m=3^k$ for some $k\in\N$, which are precisely the maps that have $\mathcal M$ as an invariant set.

The dimension of the intersection of fractal sets such as Cantor sets is an important problem popularised by  some of Furstenberg's conjectures, which in particular state that ``expansions in multiplicatively independent bases (such as 2 and 3) should have no common structure''. We refer to \cite{S16} and references therein. In order to illustrate the potential of the EI as an indicator for the relevance of the intersection of fractal sets and the compatibility of a certain dynamics with the self-similar structure of a Cantor set,  we used an estimator of the EI, introduced in  \cite{H93a}, and carried out  a numerical study in order to demonstrate its performance by comparing with our theoretical estimates.

The paper is structured as follows. In Section \ref{sec:EVL}, we introduce the framework regarding the study of rare events for dynamical systems and state general results providing conditions in order to obtain the existence of limiting laws. In Section~\ref{sec:statement-of-results}, we introduce the observables maximised on Cantor sets (the ternary Cantor set, to be more precise), define the dynamical models and state the main results of the paper. In Section~\ref{se:3x_mod1}, we state a general theorem establishing the existence of a limiting law, in the presence of clustering, for general dynamically defined Cantor sets and compatible dynamics, which allows to prove one of the main theorems stated in the previous section. In Section~\ref{sec:absence-clustering}, we prove the other main theorem stated in Section~\ref{sec:statement-of-results} regarding the existence of a limiting extreme value law, in the absence of clustering. This is the more elaborate part and includes a brief review of several tools of fractal geometry. In Section~\ref{sec:EI-indicator}, we present a numerical simulation study to illustrate the suitability of the use of the EI as an indicator of the compatibility between the dynamics and the geometrical fractal structure of the maximal set.     

\section{Laws of extreme events}
\label{sec:EVL}

Consider a discrete dynamical system $(\mathcal{X},\mathcal{B},T,\mu)$, where $\mathcal{X}$ is a compact set (an interval, in our case), $\mathcal{B}$ is the respective Borel sigma algebra, $T:\mathcal{X} \to \mathcal{X}$ is a measurable map, and $\mu$ is an invariant measure with respect to $T$. We will follow the Extreme Value approach and, therefore, we consider an observable function $\varphi:\mathcal{X}\to \R^{+}\cup\{\infty\}$ and define the stochastic process, $(X_n)_{n\in\N}$, in the following way,
\begin{equation}
\label{eq:stochastic_process}
  X_n(x)=\varphi\circ T^{n}(x).
\end{equation}
Note that the invariance of $\mu$ implies the stationarity of $(X_n)_{n\in\N}$. We are particularly interested in the extremal behaviour of such stochastic processes, which is tied to the recurrence properties of the set of global maxima of $\varphi$, as was proved in \cite{FFT10,FFT11}. We denote this set of global maximal points as $\mathcal M$, \ie we assume that there exists $Z=\max_{x\in\X}\varphi(x)$, where we allow $Z=+\infty$, and
$$\M=\{x\in \X:\; \varphi(x)=Z\}.$$
In what follows, $\zeta$ will always denote a generic point of $\mathcal M$. In this paper, most of the times, $\mathcal M=\mathcal C$, where $\mathcal C$ denotes the usual ternary Cantor set.

From the stochastic process $(X_n)_{n\in \N}$, we define the process of partial maxima $(M_n)_{n\in\N}$ whose limiting distribution we want to analyse:
\begin{equation}
\label{eq:m_n}
M_n=\text{max}\{X_0,\ldots,X_{n-1}\}.
\end{equation}

In order to study the extremal behaviour of $(X_n)_{n\in\N}$, we consider the level sets $\{X_j>u\}$, \ie the exceedances of a high threshold $u$, which correspond to the target sets in the HTS/RTS approach, and try to obtain a limit for the probability of not observing any exceedance up to a certain moment of time $m$, which depends on the level $u$. More precisely, we want to estimate $\mu(M_m\leq u)$ as $u\to Z$ or, in other words, when the target sets $\{X_j>u\}$ shrink to $\M$. In order to obtain a non degenerate limit, the dependence of $m$ on $u$ must be well tuned.

When $\mu(X_j>u)$ as a function of $u$ is not smooth, as happens here, this tuning must be performed with some care and we will use the normalisation introduced in \cite{FFT11} to deal with similar cases. Namely, we consider sequences $(w_n)_{n\in\N}$ and $(u_n)_{n\in\N}$ such that
\begin{equation}
\label{eq:w_n}
w_n\mu(X_0>u_n)\rightarrow \tau \text{  as  } n\to\infty\text{ for some $\tau\geq0$}.
\end{equation}
Then our main goal is to find some non-degenerate distribution function $H$ supported on $\R^+$ such that
\begin{equation*}
\lim_{n\to\infty}\mu(M_{w_n}\leq u_n)\rightarrow 1-H(\tau).
\end{equation*}
In \cite{FFT11} this type of distributional limit was called cylinder Extreme Value Law (EVL). When $\M$ is  a finite or countable set of points and $T$ is either a uniformly expanding map or admits an hyperbolic first return time induced map, then, as seen for example in \cite{FFT12, K12, HNT12, AFV15, FFTV16, AFFR16,AFFR17}, we have that
\begin{equation}
\label{eq:EVL}
\lim_{n\to\infty}\mu(M_{w_n}\leq u_n)\rightarrow \e^{-\theta\tau},
\end{equation}
where $0\leq\theta\leq1$. When such a limit exists, then $\theta$ is called the \emph{Extremal Index}. The EI is associated to the recurrence properties of $\M$. In fact, when $\M$ is reduced to a single point, as seen in \cite{FFT12, K12, AFV15,FFTV16}, a full dichotomy holds: either $\M=\{\zeta\}$ is non-recurrent, which translates to $\zeta$ being non-periodic, and then there is no clustering of exceedances and $\theta=1$, or $\M=\{\zeta\}$ is recurrent, \ie $\zeta$ is a periodic point, which is responsible for the appearance of clustering and an EI less than 1 (if the map is differentiable along the orbit of $\zeta$ and the invariant measure is absolutely continuous w.r.t. Lebesgue measure, we have $\theta=\frac1{|DT^p(\zeta)|}$).

\subsection{Existence of limiting laws}
\label{subsec:existence-EVL}
The main purpose of this subsection is to provide general conditions which allow us to prove the existence of a limiting law as stated in \eqref{eq:EVL}. Let $(u_n)_{n\in\N}$ and $(w_n)_{n\in\N}$ be as in \eqref{eq:w_n}.
Consider a sequence $(q_n)_{n\in\N}$ such that
\begin{equation}
\label{eq:qn-def}
\lim_{n\rightarrow\infty}q_n=\infty\qquad\mbox{and}\qquad \lim_{n\rightarrow\infty}\frac{q_n}{w_{n}}=0.
\end{equation}
Let $T^{-i}$ denote the $i$-th preimage by  the map $T$.
Fixing $u\in\R$ and $q\in\N$, we define the following events,
\begin{align}
\label{eq:U-A-def}
U(u)&:= \{X_0>u\},\nonumber\\
\AA_{q}(u)&:=U(u)\cap\bigcap_{i=1}^{q}T^{-i}(U(u)^c)=\{X_0>u, X_1\leq u, \ldots, X_{q}\leq u\}.
\end{align}
While the event $U(u)$ corresponds to the occurrence of an exceedance, the event  $\AA_{q}(u)$ corresponds to the occurrence of an exceedance which terminates a cluster of exceedances, \ie if $T^{-j}(\AA_{q}(u))$ occurs, then the next exceedance after the one observed at time $j$ must belong to a new and different cluster of exceedances. In particular, $q$ can be though as the maximal waiting time between two exceedences within the same cluster.

Let $B\in\B$ be an event. For $s,\ell\in \N$, we define,
\begin{equation*}
\mathscr W_{s,\ell}(B)=\bigcap_{i= s}^{s+\ell-1} T^{-i}(B^c).
\end{equation*}
Observe that 
$
\mathscr W_{0,n}(U(u_n))=\{M_n\leq u_n\}.
$
For each $n\in\N$, set $U_n:=U(u_n)$, $\AA_{q_n,n}:=\AA_{q_n}(u_n)$ and
\begin{equation}
\label{def:thetan}
\theta_n:=\frac{\mu\left(\AA_{q_n,n}\right)}{\mu(U_n)}.
\end{equation}
We will see that $\theta_n$ provides a good estimate for the EI. In fact, the EI, $\theta$, will be such that
\begin{equation}
\label{eq:OBriens-formula}
\theta=\lim_{n\to\infty}\theta_n.
\end{equation}
We will refer to \eqref{eq:OBriens-formula} as O'Brien's formula to compute the EI (see \cite{O87}).

We start by stating a condition that requires some sort of asymptotic independence of events when the time gap between them increases.
\begin{condition}[$\D_{q_n}(u_n, w_{n})$]\label{cond:D} We say that $\D_{q_n}(u_n, w_{n})$ holds for the stochastic process $(X_n)_{n\in\N}$ if for every  $\ell,t,n\in\N$
\begin{equation}\label{eq:D1}
\left|\mu\left( \AA_{q_n,n}\cap
 \mathscr W_{t,\ell}\left( \AA_{q_n,n}\right) \right)-\mu\left( \AA_{q_n,n}\right)
  \mu\left(\mathscr W_{0,\ell}\left( \AA_{q_n,n}\right)\right)\right|\leq \gamma(n,t),
\end{equation}
where $\gamma(n,t)$ is decreasing in $t$ for each $n$ and there exists a sequence $(t_n)_{n\in\N}$ such that $t_n=o(w_n)$ and
$w_{n}\gamma(n,t_n)\to0$ when $n\rightarrow\infty$.
\end{condition}
The next condition forbids the concentration of clusters of exceedances.
\noindent Consider the sequence $(t_n)_{n\in\N}$ given by condition  $\D_{q_n}(u_n, w_{n})$ and let $(k_n)_{n\in\N}$ be another sequence of integers such that
\begin{equation}
\label{eq:kn-sequence}
k_n\to\infty\quad \mbox{and}\quad  k_n t_n = o(w_n).
\end{equation}

\begin{condition}[$\D'_{q_n}(u_n,w_{n})$]\label{cond:D'q} We say that $\D'_{q_n}(u_n,w_{n})$
holds for the sequence $(X_n)_{n\in\N}$ if there exists a sequence $(k_n)_{n\in\N}$ satisfying \eqref{eq:kn-sequence} such that
\begin{equation}
\label{eq:D'rho-un}
\lim_{n\rightarrow\infty}\,w_{n}\sum_{j=q_n+1}^{\lfloor w_{n}/k_n\rfloor-1}\mu\left(  \AA_{q_n,n}\cap T^{-j}\left( \AA_{q_n,n}\right)
\right)=0.
\end{equation}
\end{condition}
We can now state a general result establishing the existence of a limiting extreme value law.
\begin{theorem}
\label{th:exists_evl}
Let $(X_n)_{n\in\N}$ be a stochastic process constructed as in \eqref{eq:stochastic_process}. Consider the sequences $(u_{n})_{n\in\mathbb{N}}$ and $(w_{n})_{n\in\mathbb{N}}$ satisfying \eqref{eq:w_n} for some $\tau\geq0$. Assume that conditions $\D_{q_n}(u_n, w_{n})$  and $\D'_{q_n}(u_n,w_{n})$ hold  for some $q_n\in\mathbb{N}_{0}$ satisfying \eqref{eq:qn-def}.
Moreover, assume that the sequence $(\theta_n)_{n\in\N}$ defined in \eqref{def:thetan} converges to some $0\leq\theta\leq 1$, \ie $\theta=\lim_{n\to\infty}\theta_n$. Then,
$$\lim_{n\to +\infty}\mu(M_{w_n}\leq u_n)=e^{-\theta\tau}.$$
\end{theorem}
The proof of this theorem follows from an easy adjustment of the proof of \cite[Corollary~4.1.7]{LFFF16}.

\subsection{Applications to systems with loss of memory}
One of the main advantages of the previous conditions when compared with the usual ones from the classical Extreme Value Theory is that $\D_{q_n}(u_n, w_{n})$ is easily checked for systems with nice decay of correlations, while the  classical conditions of the kind, similar to Leadbetter's $D(u_n)$ condition, require a uniform mixing which is very difficult to verify even for hyperbolic systems.

\label{subsec:memory-loss}
\begin{definition}[Decay of correlations]
\label{def:dc}
Let \( \mathcal C_{1}, \mathcal C_{2} \) denote Banach spaces of real valued measurable functions defined on \( \X \).
We denote the \emph{correlation} of non-zero functions $\phi\in \mathcal C_{1}$ and  \( \psi\in \mathcal C_{2} \) with respect to a measure $\mu$ as
\[
\cv_\mu(\phi,\psi,n):=\frac{1}{\|\phi\|_{\mathcal C_{1}}\|\psi\|_{\mathcal C_{2}}}
\left|\int \phi\, (\psi\circ T^n)\, \dif\mu-\int  \phi\, \dif\mu\int
\psi\, \dif\mu\right|.
\]
We say that the dynamical sytem $(\mathcal{X},\mathcal{B},T,\mu)$ has \emph{decay
of correlations}, with respect to the measure $\mu$, for observables in $\mathcal C_1$ \emph{against}
observables in $\mathcal C_2$ if there exists a rate function $\rho:\N\to \R$, with $$\lim_{n\to\infty}\rho(n)=0,$$ such that, for every $\phi\in\mathcal C_1$ and every
$\psi\in\mathcal C_2$, we have
 $$\cv_\mu(\phi,\psi,n)\leq \rho(n).
 $$
\end{definition}
The systems we will work with have decay of correlations of functions of Bounded Variation, which we define below, against observables in $L^1(\mu)$. \begin{definition}
\label{def:variation}
Given a potential $\psi:I\to \R$ on an interval $I$, the \emph{variation} of $\psi$ is defined as
$${\rm Var}(\psi):=\sup\left\{\sum_{i=0}^{n-1} |\psi(x_{i+1})-\psi(x_i)|\right\},$$
where the supremum is taken over all finite ordered sequences $(x_i)_{i=0}^n\subset I$.
\end{definition}

We use the norm $\|\psi\|_{BV}= \sup|\psi|+{\rm Var}(\psi)$, which makes the space of functions of Bounded Variation, $BV:=\left\{\psi:I\to \R:\|\psi\|_{BV}<\infty\right\}$, into a Banach space.

We will see that conditions $\D_{q_n}(u_n,w_{n})$ and $\D'_{q_n}(u_n,w_{n})$ follow from the above mentioned type of decay of correlations of the underlying dynamical system.
\begin{theorem}
\label{thm:main}
Let $(\mathcal{X},\mathcal{B},T,\mu)$ be a dynamical system and consider an observable $\varphi$ achieving a global maximum on a set $\mathcal M$.
Let $(X_n)_{n\in\N}$ be the stochastic process given by \eqref{eq:stochastic_process} and consider $(u_n)_{n\in\N}$, $(w_n)_{n\in\N}$  and $(q_n)_{n\in\N}$ as sequences such that \eqref{eq:w_n} and \eqref{eq:qn-def} hold. If the system has decay of correlations of observables in $\C_1$ against observables in $L^1(\mu)$ and if
\begin{enumerate}
\item $\displaystyle \lim_{n\to\infty}\|\I_{\AA_{q_n,n}}\|_{\mathcal C_1}w_{n}\rho({t_n})=0$, for some sequence $(t_n)_{n\in\N}$ such that $t_n=o(w_n)$
\item $\displaystyle \lim_{n\to\infty}\|\I_{U_n}\|_{\mathcal C_1}\sum_{j=q_{n}}^\infty \rho(j)=0$
\end{enumerate}
and if the sequence $(\theta_n)_{n\in\N}$ defined in \eqref{def:thetan} converges to some $0\leq\theta\leq 1$, then conditions $\D_{q_n}(u_n,w_n)$ and $\D'_{q_n}(u_n,w_n)$ are satisfied and
$$
\lim_{n\to\infty}\mu(M_{w_n}\leq u_n)=\e^{-\theta\tau}.
$$
\end{theorem}
\begin{remark}
\label{rem:conditions}
Note that under the assumption of summable decay of correlations against $L^1$ then hypothesis (1) implies $\D_{q_n}(u_n,w_n)$, while hypothesis (2) implies $\D'_{q_n}(u_n,w_n)$.
\end{remark}

\begin{proof}
By Theorem \ref{th:exists_evl}, we only need to check that the stochastic process $(X_n)_{n\in\N}$ satisfies conditions $\D_{q_{n}}(u_n,w_{n})$ and $\D'_{q_n}(u_{n},w_{n})$.\\
Consider $\phi=\I_{ \AA_{{q_n,n}}}$ and  $\psi=\I_{ \mathscr{W}_{t,\ell}(\AA_{q_n,n})}$ in Definition \ref{def:dc}. Then, there exists $C>0$, such that, for any positive numbers $\ell$ and $t$, we have
\begin{align*}
 |\mu(\AA_{q_n,n}\cap \mathscr{W}_{t,\ell}(\AA_{q_n,n}))
                &-\mu(\AA_{q_n,n})\mu(\mathscr{W}_{0,\ell}(\AA_{q_n,n}))| \\
                &= \left|\int_{\mathcal{X}}\I_{ \AA_{q_n,n}}\cdot (\I_{ \mathscr{W}_{0,\ell}(\AA_{q_n,n})}\circ T^t)d\mu-\int_{\mathcal{X}}\I_{ \AA_{q_n,n}}d\mu \int_{\mathcal{X}}\I_{ \mathscr{W}_{0,\ell}(\AA_{q_n,n})}d\mu \right| \\
                &\leq C\|\I_{\AA_{q_n,n}}\|_{\mathcal C_1}\rho(t).
\end{align*}
Condition $\D_{q_n}(u_n,w_{n})$ follows if there exists a sequence $(t_n)_{n\in\NN}$ such that $t_n=o(w_n)$ and
$\displaystyle\lim_{n\to\infty} \|\I_{\AA_{q_n,n}}\|_{\mathcal C_1}w_{n}\rho_{t_n}=0$, which is the content of hypothesis (1).\\

In order to prove \eqref{eq:D'rho-un}, we start by observing that 
$$
w_{n}\sum_{j=q_n+1}^{\lfloor w_{n}/k_n \rfloor}\mu\left( \AA_{q_n,n}\cap T^{-j}( \AA_{q_n,n})\right)\leq w_{n}\sum_{j=q_n+1}^{\lfloor w_{n}/k_n \rfloor}\mu\left( U_n\cap T^{-j}( U_n)\right)
$$
Then, we take $\phi=\psi=\I_{U_n}$,  in Definition \ref{def:dc}, to obtain that
\begin{align*}
\mu\left( U_n\cap T^{-j}( U_n)\right)
            &=\int_{\mathcal{X}} \phi\cdot (\phi\circ T^{j})d\mu
            \leq \left(\mu( U_n)\right)^2+ \left\| \I_{U_n}\right\|_{\mathcal C_1} \mu\left( U_n\right) \rho(j).
\end{align*}
Let $t_n$ be as above and take $(k_n)_{n\in\N}$ as in \eqref{eq:kn-sequence}. Recalling that  $\lim_{n\to\infty}w_{n}\mu(U_n)=\tau$,  
it follows that
\begin{align*}
w_{n}\sum_{j=q_n+1}^{\lfloor w_{n}/k_n \rfloor}\mu\left( U_n\cap T^{-j}( U_n)\right)
                & \le w_{n}\big\lfloor\tfrac {w_{n}}{k_n}\big\rfloor\mu\left( U_n\right)^2 +w_{n}\left\| \I_{ U_n}\right\|_{\mathcal C_1} \mu\left( U_n\right) \sum_{j=q_{n}+1}^{\lfloor w_{n}/k_n \rfloor}\rho_j \\
                & \leq \frac{w_{n}^{2}\mu(U_n)^{2}}{k_{n}}+w_{n}\left\| \I_{ U_n}\right\|_{\mathcal C_1} \mu\left( U_n\right) \sum_{j=q_{n}}^{\infty}\rho(j)\\
                &\leq \frac{\tau^2}{k_n}+\tau\left\| \I_{U_n}\right\|_{\mathcal C_1} \sum_{j=q_{n}}^{\infty}\rho(j)\xrightarrow[n\to\infty]{} 0,
\end{align*}
by choice of $k_n$ and hypothesis (2).
\end{proof}

\section{Observables with fractal maximal sets}
\label{sec:statement-of-results}

Let $\mathcal{C}$ denote the ternary Cantor set. In order to construct $\mathcal{C}$, we start by removing the middle third of the interval $\mathcal C_0:=[0,1]$ and define in this way the first approximation $\mathcal C_1$. Then, we start an iterative process where we build $\mathcal{C}_n$ by removing the middle third of each connected component of $\mathcal{C}_{n-1}$, as represented in Figure \ref{fig:cantor}. Repeating this process indefinitely, we obtain the set $\mathcal{C}=\cap_{n\geq1}\mathcal C_n$.
\begin{figure}[h]
\begin{tikzpicture}[decoration=Cantor set,line width=1mm,scale=2.5]
\draw (0,0) -- (3,0);
\node at (3.3,0) {$[0,1]$};
\draw decorate{ (0,-.5) -- (3,-.5) };
\node at (3.3,-.5) {$\mathcal{C}_1$};
\draw decorate{ decorate{ (0,-1) -- (3,-1) }};
\node at (3.3,-1) {$\mathcal{C}_2$};
\draw decorate{ decorate{ decorate{ (0,-1.5) -- (3,-1.5) }}};
\node at (3.3,-1.5) {$\mathcal{C}_3$};
\end{tikzpicture}
\caption{The construction of the ternary Cantor set. \label{fig:cantor}}
\end{figure}
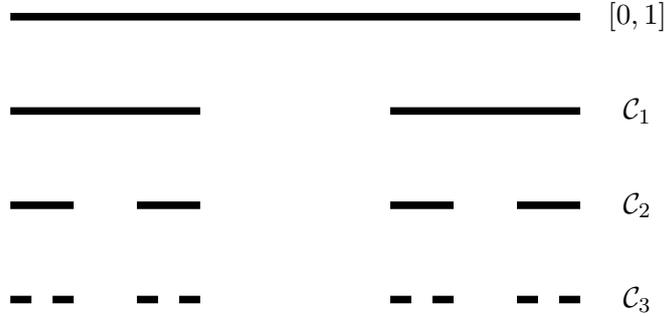

We define the observable to be the Cantor ladder function also used in \cite{MP16} as a prototype fractal landscape. Namely, for each $n\in\NN$, let $B_n:=\mathcal C_{n-1}\setminus\mathcal{C}_n$ so that $B_1=\left(\frac{1}{3},\frac{2}{3}\right)$,
$B_2=\left(\frac{1}{9},\frac{2}{9}\right)\cup \left(\frac{7}{9},\frac{8}{9}\right),\ldots$, \ie the sets $B_n$ correspond to the gaps of the Cantor set formed at the $n$-th step of its construction.
Now consider the observable
\begin{equation}
\label{eq:Cantor-ladder}
\varphi(x)=\left\{
  \begin{array}{ll}
    n, & \hbox{ if }x\in B_n,\ n=1,2,3\dots \\
    \infty, & \hbox{otherwise.}
  \end{array}
\right.
\end{equation}
\begin{figure}[h]
\includegraphics[width=0.7\textwidth]{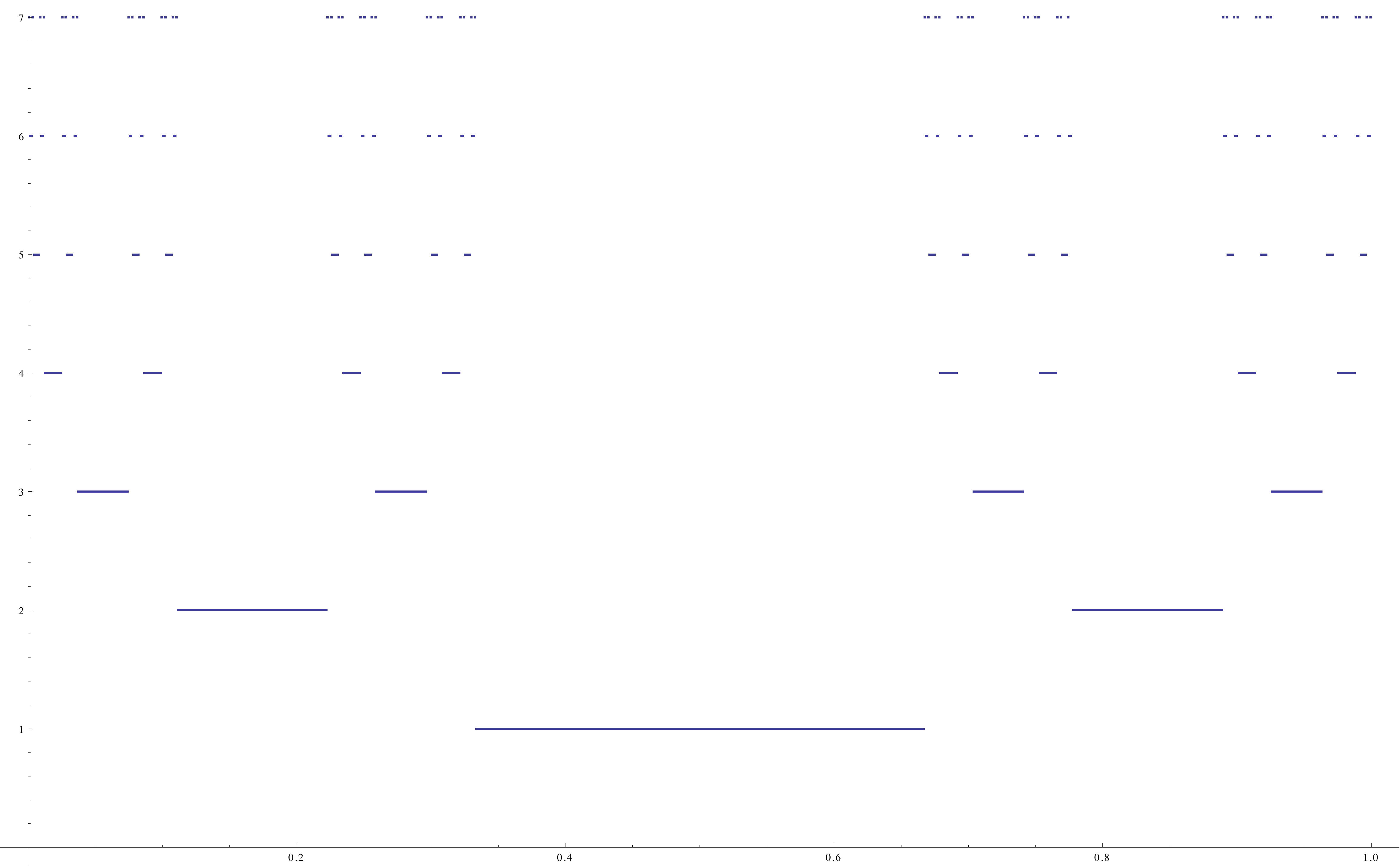}
\caption{The Cantor ladder function. \label{fig:ladder}}
\end{figure}

\begin{remark}
\label{rem:maximalset=Cantor}
Note that if $x\in \mathcal{C}$ then $x\not\in B_n$ for all $n\in\NN$, which implies that $\varphi(x)=\infty$. If $x\notin \C$ then $x\in B_n$ for some $n\in\N$ and therefore, in this case, we have that $\M=\C$.
\end{remark}

In this section, we will consider dynamical systems given by:
\begin{align}
T\colon& [0,1]\longrightarrow[0,1]\nonumber\\
&\phantomarrow{[0,1]}{x} m\cdot x \mod 1,\label{eq:dynamics}
\end{align}
where $m\in\N$. These are full branched uniformly expanding maps, which preserve Lebesgue measure (that we shall denote by $\mu$) and have exponential decay of correlations of BV observables against $L^1(\mu)$ (this follows from \cite[Corollary~8.3.1]{BG97} or \cite[Corollary~H]{AFLV11}, for example).

In \cite[Section~3]{MP16}, the authors considered the same observable $\varphi$ defined in \eqref{eq:Cantor-ladder} and the dynamics generated by an asymmetric tent map, which is also a full branched uniformly hyperbolic map, and conjectured the existence of a limiting extreme value law with an EI equal to 1, which was supported by the numerical simulations performed. We prove that when the dynamics considered is not compatible with the self-similar structure of the maximal set (which happens here when $m\neq 3^k$ for all $k\in\N$) then indeed the conjectured extreme limiting behaviour applies. To our knowledge, these are the first rigorously proved results for observables with fractal maximal sets.
\begin{theorem}
\label{thm:2x_mod1}
Let $(X_n)_{n\in\N}$ be the stochastic process given by \eqref{eq:stochastic_process} for a dynamical system $T$ defined in \eqref{eq:dynamics}, with $\N\ni m\neq3^k$ for all $k\in\N$. Consider a sequence of thresholds $(u_n)_{n\in\N}$ such  that $u_n=n$ and a sequence of times $(w_n)_{n\in\N}$ such that $w_n=\left\lfloor\tau \left(3/2\right)^{n}\right\rfloor $.
Then, condition~\eqref{eq:w_n} holds and moreover
$$ \lim_{n\to\infty} \mu(M_{w_{n}}\leq n)=\e^{ -\tau}.$$
\end{theorem}
The numerical simulations performed in \cite[Section~3]{MP16} also showed that, interestingly, the same limiting laws seem to apply when the dynamics is replaced by that of irrational rotations. The fact that these ergodic maps are not mixing and yet the agreement was still good, lead the authors of \cite{MP16} to conjecture that the role of the fast decay of correlations in assuring the validity of conditions such as $\D$ and $\D'$ was played, in this situation, by the complexity of the observable function. 
We also remark that, in all numerical studies performed in \cite{MP16} with observables maximised on fractal sets (with strictly positive Hausdorff dimension), the observed EI was always 1.
In our results we use heavily the excellent mixing properties of all the systems considered. However, not only we provide examples where the EI is strictly less than 1, as we explain how the EI is related with the compatibility between the dynamics and the fractal structure of the maximal set.
\begin{theorem}
\label{thm:3kx_mod1}
Let $(X_n)_{n\in\N}$ be the stochastic process given by \eqref{eq:stochastic_process} for a dynamical system $T$ defined in \eqref{eq:dynamics}, with $m=3^k$ for some $k\in\N$. Consider a sequence of thresholds $(u_n)_{n\in\N}$ such  that $u_n=n+k-1$ and a sequence of times $(w_n)_{n\in\N}$ such that  $w_n=\left\lfloor\tau \left(3/2\right)^{n+k-1}\right\rfloor$. Then, condition~\eqref{eq:w_n} holds and moreover
$$ \lim_{n\to\infty} \mu(M_{w_{n}}\leq n)=\e^{-\left(1-\frac{2^k}{3^k}\right) \tau}.$$
\end{theorem}
In fact, in Section~\ref{se:3x_mod1}, we prove Theorem~\ref{thm:3kx_mod1}  as a corollary of Theorem~\ref{thm:3kx_mod1-geral} which applies to more general Cantor sets. We note that, in the context of both Theorems~\ref{thm:3kx_mod1} and \ref{thm:3kx_mod1-geral}, the compatibility between the dynamics and the maximal set becomes obvious when we observe that $T(\M)=\M$, which means that $T$ preserves the structure of the Cantor sets, which play the role of a periodic point in the context of when $\M$ is reduced to a single point. The proof will follow more or less the same strategy used in \cite{FFT12} and generalised later in \cite{AFFR16,AFFR17}, which basically exploits the periodicity of the maximal set in order to be able to compute the EI from the O'Brien's formula \eqref{eq:OBriens-formula} and then verify the conditions $\D_{q_n}(u_n,w_n)$ and $\D'_{q_n}(u_n,w_n)$ that were designed to be easily checked from the excellent mixing properties of the system. 

When $m\neq3^k$ for all $k\in\N$, although $T^{j}(\M)\neq\M$ for all $j\in\N$, one can easily check that, most of the times, we have that $T^{-j}(\M)\cap\M\neq\emptyset$ and this was enough to create clustering when $\M$ was a finite or countable set (see \cite{AFFR16,AFFR17}). However, here, the maximal set has a much more complex structure and one needs to evaluate how relevant the intersections $T^{-j}(\M)\cap\M\neq\emptyset$ are when compared with $\M$ itself, which translates to how compatible the dynamics of $T$ is with the fractal structure of $\M$. We will see that since $\C$ has a thickness not less than 1 (\ie the extractions in the construction of the Cantor set are relatively not too large), then the relevance of the intersection (or the compatibility between $T$ and $\M$) can be measured by the box dimension of the intersection $T^{-j}(\M)\cap\M$, when compared with the box dimension of $\M$ itself. We will show that the box dimension of $T^{-j}(\M)\cap\M$ is strictly less than that of $\M$ (Proposition~\ref{prop:dimension}), which means that the possible clustering created by the fact that $T^{-j}(\M)\cap\M\neq\emptyset$ is negligible and, in the limit, the EI is still 1. The computation of the EI is much more subtle and we need results from fractal geometry in order to compute the dimension of such intersections and then we need to study its impact on O'Brien's formula \eqref{eq:OBriens-formula}, for which we will perform a finer analysis, where we use of the notion of thickness of dynamically defined Cantor sets introduced by Newhouse in \cite{N79}. This will be done in Section~\ref{sec:absence-clustering}.

\section{The appearance of clustering with fractal landscapes}
\label{se:3x_mod1}

When the dynamics is compatible with the self-similar structure of the fractal maximal set, we observe the appearance of clustering and a limiting law with a non-trivial EI. In this context, we consider more general fractal sets. Namely, we will consider Cantor sets generated by an \emph{Iterated Function System} (IFS) satisfying some regular conditions. These Cantor sets can also be identified as the survivor sets for some conveniently chosen dynamical systems, which will also provide a common ground to assess the compatibility of the self-similarity structure with the original dynamics. We will start by providing a description of these more general dynamically defined Cantor sets. Then, we establish the existence of a limiting law with a non-trivial EI when the dynamics is compatible with the system generating the Cantor set and, finally, we apply it to the usual ternary Cantor set. 

\subsection{Dynamically defined Cantor sets}
\label{subsec:dynamic-Cantor}

We start with a description of a class of more general Cantor sets.

Let $r\in\N$ and $\mathscr I=\{f_1,\ldots,f_s\}$ be a regular finite family of normalised contractions defined on $[0,1]$, \ie each $f_i:[0,1]\to[0,1]$ is a $C^1$ diffeomorphism such that $$|f_i(x)-f_i(y)|\leq \lambda_i |x-y|, \quad \text{for some $\lambda_i<1$}$$ and the sets $J_i=f_i([0,1])$, for $i=1,\ldots,s$, are pairwise disjoint. $\mathscr I$ is in particular an \emph{Iterated Function System} (IFS). An atractor for $\mathscr I$ is the only compact subset, $\Lambda$, of $[0,1]$, such that $\Lambda=\cup_{i=1}^s f_i(\Lambda)$. For such an IFS  there exists a unique attractor $\Lambda$, whose Hausdorff and box dimensions (see Definitions~\ref{def:box-dim} and \ref{def:Hausdorff-dim} below) are both equal to $d$, where $\sum_{i=1}^s \lambda_i^d=1$.  See \cite[Chapter~9]{F03} for proofs and more details on the subject. The attractor $\Lambda$ can be seen as a dynamically defined Cantor set, \ie $\Lambda$ can be identified as the survivor set of the dynamical system $G:\R\to\R$ defined by 
$$
G(x)=\left\{
  \begin{array}{ll}
    f_i^{-1}(x), & \hbox{ if }x\in J_i \\
   2, & \hbox{ otherwise}
  \end{array}.
\right.
$$
Namely,
$$
\Lambda=\{x\in[0,1]\colon\; G^n(x)\in[0,1],\; \mbox{for all}\;n\in\N\}.
$$
Let $\Lambda_0=[0,1]$ and for all $n\in\N$ set
\begin{equation*}
\Lambda_n=G^{-1}(\Lambda_{n-1})=\{x\in[0,1]\colon T^j(x)\in[0,1],\;\mbox{for all $j=1,\ldots,n$}\}.
\end{equation*}
\begin{remark}
\label{rem:iterates-property}
Note that  $\Lambda=\cap_{n\geq0}\Lambda_n$ and, for all $j\in\N$, we have $G^{-j}(\Lambda_n)=\Lambda_{n+j}$ because if $G(x)\notin[0,1]$ then $G^j(x)\notin[0,1]$ for all $j\in\N$.
\end{remark}

\subsection{Laws of rare events for systems compatible with dynamically defined Cantor sets}
\label{subsec:EVL-general-Cantor-sets}
We adapt the definition of the observable function $\varphi:[0,1]\to \R$ so that, in this case, the maximal set is $\Lambda$. Namely, we set
$$
\varphi(x)=\left\{
  \begin{array}{ll}
    n, & \hbox{ if }x\in \Lambda_n\setminus\Lambda_{n+1},\ n=1,2,3\dots \\
    \infty, & \hbox{ if }x\in \Lambda
  \end{array}
\right.
$$

Now we define a dynamical system which is compatible with the dynamics that generated $\Lambda$, namely, we define $F:[0,1]\to[0,1]$ by $F(x)=G(x)$ for all $x\in\cup_{i=1}^s J_i$ and if $I$ denotes a connected component of $[0,1]\setminus \cup_{i=1}^s J_i$ then, on $I$, we define $F$ as a linear map so that $F$ maps $I$ onto $[0,1]$. Note that $F$ is a piecewise uniformly expanding map and therefore admits an invariant absolutely continuous probability measure $\mu$. Moreover, from \cite[Corollary~8.3.1]{BG97}, it follows that $F$ has exponential decay of correlations of BV observables against $L^1$, \ie for all $\phi\in BV$ and $\psi\in L^1(\mu)$, there exist $C>0$  and $0<r<1$ such that
\begin{equation}
\label{eq:DC-F}
\cv_\mu(\phi,\psi,n)\leq C r^n.
\end{equation}

\begin{theorem}
\label{thm:3kx_mod1-geral}
Let $(X_n)_{n\in\N}$ be the stochastic process given by \eqref{eq:stochastic_process} for the dynamical system $T=F^k$, for some $k\in\N$, where $F$ and the observable $\varphi$ are as defined just above. Consider a sequence of thresholds $(u_n)_{n\in\N}$ such  that $u_n=n+k-1$, a sequence of times $(w_n)_{n\in\N}$ such that  $w_n=\left\lfloor\tau (\mu(\Lambda_{n+k-1}))^{-1}\right\rfloor$. 

Assume that there exist a sequence $(t_n)_{n\in\N}$, with $t_n=o(w_n)$, and a sequence $(q_n)_{n\in\N}$, as in \eqref{eq:qn-def}, such that:   
$\lim_{n\to\infty}\left\| \I_{\AA_{q_n,n}}\right\|_{BV}   w_{n} r^{t_n}=0$ and $\lim_{n\to\infty}\left\| \I_{U_n}\right\|_{BV}r^{q_n}=0$. 

Assume moreover that there exists $0\leq\theta\leq1$ such that
$$
\theta=\lim_{n\to\infty}\frac{\mu(\Lambda_{n+k-1}\setminus\Lambda_{n+2k-1})}{\mu(\Lambda_{n+k-1})}.
$$

Then, 
$$ \lim_{n\to\infty} \mu(M_{w_{n}}\leq n)=\e^{-\theta \tau}.$$
\end{theorem}

\begin{proof}
We start by noting that for the sequence of thresholds $u_{n}=n+k-1$, we have $U_n=\Lambda_{n+k-1}$ and then the definition of $w_n$ makes condition \eqref{eq:w_n} trivially satisfied. 

Next, we observe that the compatibility between $F$ and $G$ allows for a simple characterisation of the sets $\AA_{q_n,n}$. 
We claim that 
$$
 \AA_{q_n,n}=\Lambda_{n+k-1}\setminus \Lambda_{n+2k-1}.
$$
To check this claim we start by proving that for all $j\leq n/k$, we have $$ T^{-j}(\Lambda_n)\cap \Lambda_n=\Lambda_{n+kj}.$$ Clearly, $\Lambda_{n+kj}\subset  T^{-j}(\Lambda_n)\cap \Lambda_n$. For the other inclusion, consider that $x\in T^{-j}(\Lambda_n)\cap \Lambda_n$. Since $x\in\Lambda_n$, then $G^i(x)\in[0,1]$ for all $i=1,\ldots,n$ and since $j\leq n/k$ then $T^j(x)=F^{jk}(x)=G^{jk}(x)\in[0,1]$. Since $x\in T^{-j}(\Lambda_n)$, then $G^{jk}(x)\in\Lambda_n$, which means that $G^i(x)\in[0,1]$ for all $i=1,\ldots,n+jk$ and therefore $x\in \Lambda_{n+kj}$. 
Now, observing that $\Lambda_n^c\subset \Lambda_{n+1}^c$ for all $n\in\N$ and recalling the definition of $ \AA_{q_n,n}$ we obtain:
$$
  \AA_{q_n,n}=\bigcap_{i=1}^{q_{n}} T^{-i}(\Lambda_{n+k-1}^{c})\cap \Lambda_{n+k-1}=\Lambda_{n+2k-1}^{c}\cap \Lambda_{n+k-1}=\Lambda_{n+k-1}\setminus\Lambda_{n+2k-1}. 
$$ 
The fact that $F$ has decay of correlations of BV against $L^1$ as expressed in \eqref{eq:DC-F} together with the assumptions $\lim_{n\to\infty}\left\| \I_{\AA_{q_n,n}}\right\|_{BV}   w_{n} r^{t_n}=0$ and $\lim_{n\to\infty}\left\| \I_{U_n}\right\|_{BV}r^{q_n}=0$ guarantee that conditions (1) and (2) from Theorem~\ref{thm:main} hold. Moreover, the assumption on $\theta$ gives that $\theta_n=\mu( \AA_{q_n,n})/\mu(U_n)\xrightarrow[n\to\infty]{}\theta$. Consequently, the result follows from direct application of Theorem~\ref{thm:main}.
\end{proof}
\subsection{Application to the ternary Cantor set}
\label{subsec:clustering-ternary}
We  apply Theorem~\ref{thm:3kx_mod1-geral} to the ternary Cantor and prove Theorem~\ref{thm:3kx_mod1}.
\begin{proof}[Proof of Theorem~\ref{thm:3kx_mod1}]
We start by checking the hypothesis of Theorem~\ref{thm:3kx_mod1-geral} and then verify the formula provided for the Extremal Index $\theta$.
In this case, the IFS is given by $f_1(x)=1/3x$ and $f_2(x)=1/3x+2/3$, the map $F$ is given by $F(x)=3x \mod 1$, the cantor set $\Lambda=\mathcal C$ and $\Lambda_n=\mathcal C_n$. The invariant measure $\mu$ is Lebesgue measure and the rate of decay of correlations expressed in \eqref{eq:DC-F} is such that $r=1/3$. We set $q_n=\floor{(n+k-1)/k}$ and observe that $U_n=\mathcal C_{n+k-1}$ and $\AA_{q_n,n}=\mathcal{C}_{n+k-1}\setminus \mathcal{C}_{n+2k-1}$. Since $\mathcal C_n\subset \mathcal C_{n-1}$, for all $n\in\N$, and $\mu(\mathcal C_{n})=(\frac23)^n$, then $w_n=\left\lfloor\tau \left(3/2\right)^{n+k-1}\right\rfloor$,  $q_n=o(w_n)$ and we obtain:
\begin{equation*}
\mu(\AA_{q_n,n})=\left(\frac23\right)^{n+k-1}-\left(\frac23\right)^{n+2k-1}=\left(1-\frac{2^k}{3^k}\right)\left(\frac{2}{3}\right)^{n+k-1}
\end{equation*}
and, moreover,
$$\|\I_{U_{n}}\|_{BV}\leq2^{n+k+1},\qquad\|\I_{\AA_{q_n,n}}\|_{BV}\leq2^{n+2k}+1\leq2^{n+2k+1}.  $$
Let $t_{n}=n^{2}$ and note that clearly
$t_n =o(w_{n})$.
Since $r=1/3$, then
\begin{align*}
    \lim_{n\to\infty}\left\| \I_{\AA_{q_n,n}}\right\|_{BV}   w_{n} r^{t_n}
    &\leq \lim_{n\to\infty}  \left\lfloor\tau \left(3/2\right)^{n+k-1}\right\rfloor 2^{n+2k+1}  r^{n^{2}} \\
        &\leq 2\lim_{n\to\infty}  2^{k-1}\tau 3^{n+k-1}  \left(1/3^{k} \right) ^{n^{2}} + 2^{n+2k}\left(1/3^{k} \right) ^{n^{2}} = 0.
\end{align*}
Moreover, there exists some constant, $C'>0$, such that
\begin{align*}
  \lim_{n\to\infty}\left\| \I_{U_{n}}\right\|_{BV} r^{q_n}
                  &= \lim_{n\to\infty} 2^{n+k+1} r^{q_n}\leq\lim_{n\to\infty} 2^{n+k+1} (1/3^k)^{n/k+1-1/k}\leq C'\lim_{n\to\infty} \left( 2/3 \right)^{n} =0. 
\end{align*}
Finally, we use O'Brien's formula to compute the EI:
$$
\lim_{n\to\infty}\theta_n=\lim_{n\to\infty}\frac{\mu(\AA_{q_n,n})}{\mu(U_n)}=\lim_{n\to\infty}\frac{\left(1-\frac{2^k}{3^k}\right)\left(\frac{2}{3}\right)^{n+k-1}}{\left(\frac{2}{3}\right)^{n+k-1}}=\left(1-\frac{2^k}{3^k}\right)=:\theta.
$$
As a consequence of Theorem~\ref{thm:3kx_mod1-geral}, we obtain
$\lim_{n\to\infty} \mu(M_{w_{n}}\leq n)=\e^{-\left(1-\frac{2^k}{3^k}\right) \tau}  .
$
\end{proof}

\section{The absence of clustering for fractal maximal sets}
\label{sec:absence-clustering}

As we mentioned earlier, when there is no compatibility between the dynamics and the self similarity structure of the fractal maximal set, then no clustering of rare events is expected. This compatibility is related to the significance of the intersections between the maximal set and its iterates. The significance will be measured by the box dimension of those intersections and therefore we start in Section~\ref{subsec:fractal-preliminaries} by recalling some techniques that we will use in order to estimate the dimension of the referred intersections, which will be carried out in Section~\ref{subsec:box-dimension-estimates}. Then, in Section~\ref{subsec:dimension-EI}, we translate the significance of the intersection expressed in terms of box-dimension into the relevance of the measure of $\AA_{q_n,n}$ when compared with the measure of $U_n$. This will be done using the notion of \emph{thickness} used by Newhouse in \cite{N79}. Finally, in Section~\ref{subsec:D-D'}, we prove conditions $\D$ and $\D'$, in order to conclude the proof of Theorem~\ref{thm:2x_mod1}.

\subsection{Preliminaries and notions from Fractal Geometry}
\label{subsec:fractal-preliminaries}

One of the main difficulties to prove the existence of extreme value laws for stochastic processes arising from observables maximised on Cantor sets is to calculate the EI $\theta$ based on O'Brien's formula. In the next sections, we will present a technique to calculate the EI based on the box dimension of the sets $T^{-j}(\C)\cap \C$.
To do this, we will need a construction given in \cite{M08}, where the author uses Digraph Iterated Function Systems, introduced in \cite{MW88}, in order to describe the intersection of fractal sets. Then, we will combine a result from \cite{MW88} to estimate the Hausdorff dimension of such intersections with a result from \cite{DN04}, which relates the respective Hausdorff and box dimensions, so that we obtain an estimate for the box dimension of the intersections $T^{-j}(\C)\cap \C$, which we will, ultimately, use later to compute the EI.

We start by recalling some notions of Fractal Geometry (referring to the book \cite{F03} for further details) and, in particular, the concept of Digraph Iterated Function System used by McClure to describe the intersection of fractal sets in \cite{M08}.
\begin{definition}
[Box Dimension]
\label{def:box-dim}
Let $F$ be a subset of $\mathbb{R}^{d}$, then, the box dimension of $F$ is defined as
\begin{equation}
\label{eq:box_dimension}
\text{dim}_{B}(F)=\lim_{\varepsilon\to 0} \frac{\log N_{\varepsilon}(F)}{-\log \varepsilon},
\end{equation}
where $N_{\epsilon}(F)$ denotes the smallest number of balls of radius $\varepsilon$ that cover $F$, whenever the limit exists. The upper and lower box dimension are defined by taking the $\limsup$ and $\liminf$, respectively, in the previous limit.
\end{definition}

\begin{definition}
[Hausdorff Dimension]
\label{def:Hausdorff-dim}
Let $F$ be a subset of $\mathbb{R}^{d}$ and $\{F_i\}_{i\in\N}$ be a countable collection of sets, with diameter at most $\delta$, that cover $F$. For $\alpha\geq 0$, we define the $\alpha$ - dimensional Hausdorff measure of $F$ as
$$ H^{\alpha}(U) = \lim_{\delta\to 0} \text{inf}\left\{ \sum^{\infty}_{i=1} \lvert F_i\rvert^{\alpha} : \text{where } \{F_i\} \text{ is a } \delta-\text{cover of } F  \right\}.$$
The Hausdorff dimension of $F$ is defined as
$$ \text{dim}_{H}(F)=\text{inf}\lbrace \alpha:H^{\alpha}(F)=0\rbrace=\text{sup}\lbrace \alpha:H^{\alpha}(F)=\infty\rbrace.$$
\end{definition}

Both definitions of dimension are finitely stable, \ie if $\{F_1,\ldots,F_n\}$ is a finite collection of subsets of $\mathbb{R}^{d}$, then
$$ \text{dim}_{B}\left(\bigcup^{n}_{i=1}F_i\right)=\underset{i}{\text{max }} \text{dim}_{B}(F_i) \qquad\mbox{and}\qquad \text{dim}_{H}\left(\bigcup^{n}_{i=1}F_i\right)=\underset{i}{\text{max }} \text{dim}_{H}(F_i).$$
\begin{remark}
\label{rem:Cantor-dimension}
We note that the box dimension of the ternary Cantor set is $\text{dim}_{H}(\C)=\text{dim}_{B}(\C)=\log2/\log3$. Moreover, since $\C$ can be generated from an Iterated Function System (IFS), which satisfies the so called \emph{open set condition}, then the box dimension of $\C$ coincides with its Hausdorff dimension. We defer to \cite{F03} for definitions and proofs of these statements.
\end{remark}

The notion of Digraph Iterated Function Systems (Digraph IFS) generalizes the most common setup of IFS. We follow closely the notation and presentation in \cite{M08}.
\begin{definition}[Digraph IFS]
A Digraph IFS consists of a digraph $G$ where the set of vertices is denoted by $V$ and the set of edges is denoted by $E$. To each of the vertices, we associate a metric space $X_v$. Furthermore, to each of the edges between two vertices $u$ and $v$, denoted by $e\in E_{uv}$, we associate a similarity $f_e:X_v\to X_u$ with ratio $r_e$.
For every path $\alpha$ in the graph $G$, we form the function $f_{\alpha}$ by composing the functions $f_{e}$ along the path in reverse order. The ratio $r_{\alpha}$ of $f_{\alpha}$ is just the product of the ratios of the composed functions.
If every $r_{\alpha}$ is less than one, then, there exists a set, $W$, which is a union of compact sets $W_v$, one for every vertex, such that for every $u\in V$,
\begin{equation}
W_u =\bigcup_{v\in V}\bigcup_{e\in E_{uv}} f_e(W_v).
\end{equation}
This invariant set, $W$, is called the attractor of the Digraph IFS. The existence of this set $W$ is guaranteed if the similarities $f_{e}$ have ratios smaller than one (see \cite{E08}).
\end{definition}

It is possible to represent a Digraph IFS in matrix notation. For that purpose, we construct a Digraph IFS matrix, $M^{*}$, with entries indexed by $(u,v)\in V\times V$. The value of each entry will be the set of edges that link one vertex to another. To a Digraph IFS, $G$, we also associate a Digraph IFS substitution matrix, $M$, which is just the adjacency matrix of digraph $G$.

If $E$ is an attractor of a standard IFS and $g$ is a bijection, then $E\cap g(E)$ can be represented as an attractor of a Digraph IFS. Namely,
\begin{theorem}[Theorem~1 of \cite{M08}]
\label{th:maclure}
Let $E$ be an attractor of an IFS, $\{f_i\}_{i=1}^{m}$, such that all functions $f_i$ are bijective contractions. Assume that there exists a finite set of bijections, $S$, such that, for all $g\in S$ satisfying $E\cap g(E)\neq \emptyset$ and for all $i,j=1,\ldots,m$ satisfying $E\cap f_{i}^{-1}gf_j(E)\neq \emptyset$, then $f_{i}^{-1}gf_j\in S $.
Under this condition the list of sets $ \left\{E\cap g(E):g\in S\right\} $ forms the attractor of a Digraph IFS.
\end{theorem}
In order to construct the Digraph IFS whose attractor is $E\cap g(E)$, it is necessary to use an iterative process to find the set of functions $S$. We start with a set $S_0=\left\{ g\right\}$, then we define the set $$S_{k+1}=S_k\cup \left\{ f_{i}^{-1}hf_{j} : h\in S_k \text{ and } i,j=1,\ldots,m\right\}.$$
To fulfil the hypotheses of Theorem~\ref{th:maclure}, in each step, we select only those functions, $h$, such that $E\cap h(E)\neq \emptyset$. We continue the procedure until no new function is found. The functions in $S$ will work as the vertices of the Digraph IFS while the edges will be labelled by the functions $f_i$. So, each row and line of the matrix $M^{*}$ have an associated function that belongs to the set $S$. Each entry of this matrix can be represented by a pair of functions $(g,h)\in S\times S$. Each of the entries $(g,h)$ will be a finite set of functions
\begin{equation*}
\left\{f_1,f_2,\dots, f_k \right\},
\end{equation*}
whose cardinality is the number of directed edges from $g$ to $h$, where a function $f_i$ belongs to this set if and only if
$$h=f_{i}^{-1}gf_j,$$
for some $j$.

The substitution matrix, $M$, of the Digraph IFS, will be the matrix $M^{*}$ but with each entry replaced by the cardinality of the corresponding set. 

\begin{definition}[Open Set Condition]
\label{de:opensetcondition}
A Digraph IFS satisfies the open set condition if and only if there exists open sets $\Omega_v\in X_v$, such that, for every $u,v\in V$ and $e\in E_{uv}$,
$$f_e(\Omega_v)\subseteq \Omega_u  $$
and for all $u,v,v'\in V$, $e\in E_{uv}$ and $e'\in E_{uv}$ with $e'\neq e$,
$$ f_e(\Omega_v)\cap f_{e'}(\Omega_{v'})=\emptyset .$$
\end{definition}

From \cite{MW88}, one has that the Hausdorff dimension of  the attractor $W$ of a Digraph IFS satisfying the open set condition (which, by \cite{DN04}, under certain conditions that are verified in our setting, is equal to its box dimension) can be written in terms of the spectral radius of $M$ and  the common ratios of the similarities of the Digraph IFS. For this reason, we recall here the definition and some useful properties of the spectral radius of a matrix. Let $A\in \M^n(\mathbb{C})$ be a complex matrix. The spectral radius of $A$ is defined as
$$\rho(A)=\text{max}\left\{\lvert\lambda\rvert:\lambda \text{ is an eigenvalue of } A \right\}.$$
The Euclidean norm of a complex matrix $A\in \M^n(\mathbb{C})$ is defined as
\begin{equation*}
\left\lVert A \right\rVert_{2} = \underset{x\neq 0}{\text{sup}} \frac{\left\lVert Ax \right\rVert_{2}}{\left\lVert x \right\rVert_{2}},
\end{equation*}
where $\left\lVert x \right\rVert_{2}$ is the usual Euclidean vector norm.
This norm is multiplicative (in some literature also called consistent or sub-multiplicative) in the sense that satisfies $\left\lVert AB \right\rVert_{2}\leq \left\lVert A \right\rVert_{2}\left\lVert B \right\rVert_{2}$ for arbitrary matrices $A$ and $B$.  Moreover, we have (see  \cite{DP65}, for example):
\begin{equation}
\label{eq:raio-norma}
 \rho(A) \leq \left\lVert A \right\rVert_{2}.
 \end{equation}
Assume that $A$ is a nonnegative matrix, \ie every entry is either positive or $0$. If  $B$ is a principal submatrix of $A$, then the spectral radius of $B$ satisfies (see \cite{BP94}, for example):
\begin{equation}
\label{eq:raio-sub-raio}
\rho(B)\leq\rho(A).
\end{equation}

We end this subsection by stating an inequality that will become very useful later in the estimation of the spectral radius of the adjacency matrix that we will construct.
\begin{proposition}
\label{prop:peter_paul}
Let $a,b$ by any positive real numbers and consider $\varepsilon>0$. Then,
$$2ab\leq \frac{a^2}{\varepsilon}+\varepsilon b^{2}. $$
\end{proposition}

\subsection{Intersection of Fractal Sets}
\label{subsec:box-dimension-estimates}
We will now apply the procedure described in Section~\ref{subsec:fractal-preliminaries} (see \cite[Section~2.2]{M08} for further details) to estimate the box dimension of the set $T^{-q}(\mathcal{C})\cap \mathcal{C}$. Namely, our main goal is to show:
\begin{proposition}
\label{prop:dimension}
Let $T=mx\mod1$, where $\N\ni m\neq 3^k$ for any $k\in\N$. Then, for all $q\in\N$, we have:
$$
\emph{dim}_H(T^{-q}(\mathcal{C})\cap\mathcal{C})=\emph{dim}_B(T^{-q}(\mathcal{C})\cap\mathcal{C})\leq\frac12.
$$
\end{proposition}

The rest of this subsection is dedicated to the proof of Proposition~\ref{prop:dimension}.  We start by noting that, for any $q$ integer,
$$ T^{q}(x)=m^{q}x \mod 1 .$$
Therefore, the set $T^{-q}(\mathcal{C})$ is a union of sets formed by taking the preimage of $\mathcal{C}$ by each one of the branches of $T^{q}.$ This implies that the functions $g$ of interest to us to start the algorithm described in Theorem \ref{th:maclure} are of the form
\begin{equation}
\label{eq:form_g}
g=\frac{1}{m^q}x+b_g,
\end{equation}
where $b_g$ is of the form $k/(m^q)$ with $k$ an integer less than $m^q$.
The algorithm leads to the construction of $m^{q}$ sets of functions, which we denote by $S_{q}^{k}$ for $k\in\{0,\ldots,m^{q}-1\}$, depending on the constant term of the function $g$ that initiates the algorithm. Each of the sets $S_{q}^{k}$ yields a substitution matrix, $M^{k}_{q}$, associated with the respective Digraph IFS.

Define the functions $f_i=x/3 +b_i$, where $b_i$ is either $0$ or $2/3$. This set of functions forms an IFS whose attractor is the ternary Cantor set $\mathcal{C}$.

For a given $q$ and $k$, the functions $h=f_{i}^{-1}gf_j$ that belong to the set $S_{q}^{k}$, are of the form,
$$ f_{i}^{-1}gf_j=\frac{1}{m^q}x+3\left(\frac{1}{m^q}b_j+b_g-b_i\right) ,$$
where $g$ already belongs to the set in question. Hence, for $h$ to belong to $S_{q}^{k}$, it is necessary that its constant term satisfies
\begin{equation}
\label{eq:belong_s}
3\left(\frac{1}{m^q}b_j+b_g-b_i\right)\in \left\{\frac{-1}{m^q},0,\frac{1}{m^q},\frac{2}{m^q},\ldots,1\right\}.
\end{equation}
Therefore, all the functions in $S^{k}_{q}$ are of the form,
$$ \frac{1}{m^q}x+\frac{s}{m^q},$$
where $s\in\{-1,0,\ldots,m^{q}\}$.

For better understanding, we divide the characterization of the matrices $M_{q}^{k}$ into the following smaller results.
\begin{lemma}
\label{le:lemma1}
Let $q\in\N_0$ and $k\in\{0,\ldots,m^{q}-1\}$, then, every entry of the matrix $M_{q}^{k}$  is either $0$ or $1$.
\end{lemma}
\begin{proof}
Fix $q$, $k$ and let $g$ be a function in $S_{q}^{k}$. As seen in \eqref{eq:belong_s}, any other function $h$ that belongs to the set $S_{q}^{k}$
must be equal to
$$ h=\frac{1}{m^q}x+b_h ,$$
where $b_h$ is of the form $s/(m^q)$ with $s\in\{-1,\dots,m^q\}$. To prove the Lemma, we will need to address two different cases, each with two different possibilities:
\begin{itemize}
  \item If $h=f_{1}^{-1}gf_2$ then  $h\neq f_{2}^{-1}gf_1,$\\
  \item If $h=f_{1}^{-1}gf_2$ then  $h\neq f_{2}^{-1}gf_2,$\\
  \item If $h=f_{1}^{-1}gf_1$ then  $h\neq f_{2}^{-1}gf_1,$\\
  \item If $h=f_{1}^{-1}gf_1$ then  $h\neq f_{2}^{-1}gf_2.$\\
\end{itemize}
For the first case, assume that  $h=f_{1}^{-1}gf_2$ and $h=f_{2}^{-1}gf_1$. Then, we would obtain
$$ \frac{1}{m^q}x+3\left(\frac{1}{m^q}b_1+b_g-b_2\right)=\frac{1}{m^q}x+3\left(\frac{1}{m^q}b_2+b_g-b_1 \right). $$
Since $b_1=0$ and $b_2=2/3$, we are led to $-1=\frac{1}{m^{q}}, $
which is an absurd.

Consider that $h=f_{1}^{-1}gf_2$ and $h=f_{2}^{-1}gf_2$.
Then,
$$ \frac{1}{m^q}x+3\left(\frac{1}{m^q}b_1+b_g-b_2\right)=\frac{1}{m^q}x+3\left(\frac{1}{m^q}b_2+b_g-b_2 \right), $$
and $2/m^{q}=0$, which is an absurd.

Consider that $h=f_{1}^{-1}gf_1$ and $h=f_{2}^{-1}gf_1$.
Then,
$$ \frac{1}{m^q}x+3\left(\frac{1}{m^q}b_1+b_g-b_1\right)=\frac{1}{m^q}x+3\left(\frac{1}{m^q}b_2+b_g-b_1 \right), $$
and $2/m^{q}=0$ which is, again, an absurd.

For last case, assume that $h=f_{1}^{-1}gf_1$ and $h=f_{2}^{-1}gf_2$ and observe that
$$ \frac{1}{m^q}x+3\left(\frac{1}{m^q}b_1+b_g-b_1\right)=\frac{1}{m^q}x+3\left(\frac{1}{m^q}b_2+b_g-b_2 \right)$$
implies $1=1/m^q$, which is an absurd and the Lemma is proved.
\end{proof}

\begin{lemma}
\label{le:lemma2}
Let $q\in\N_0$ and $k\in\{0,\ldots,m^{q}-1\}$, then the sum of the elements of each row of the matrix $M_{q}^{k}$ is at most $2$.
\end{lemma}
\begin{proof}
Consider a function $g$ in $S_{q}^{k}$.
As seen before, any other function $h\in S_{q}^{k}$ must be of the form
$$ h=\frac{1}{m^q}x+b_h ,$$
where $b_h=s/(m^q)$ with $s\in\{-1,\ldots,m^q\}$.
According to the possible values of $b_i$ and $b_j$ there are four different possibilities for the line of $M_{q}^{k}$ indexed by $g$ to have entries equal to $1$. We will prove that these cases form two disjoint groups of two elements, which will prove the claim of the Lemma.

Assume that $b_i=0$ and $b_j=0$ and that $h=f_{i}^{-1}gf_j$  belongs to $S_{q}^{k}$.
By \eqref{eq:belong_s}, the constant term of $h$ satisfies
\begin{equation}
\label{eq:lema2_2}
3b_g \in \left\{\frac{-1}{m^q},0,\frac{1}{m^q},\frac{2}{m^q},\ldots,1\right\}.
\end{equation}
By contradiction, assume that $h^{*}=f_{i}^{-1}gf_j$ belongs to $S_{q}^{k}$ with $b_i=2/3$ and $b_j=0$.
Again, by \eqref{eq:belong_s}, the constant term of $h^{*}$ satifies
\begin{equation}
\label{eq:lema2_3}
3b_g-2 \in \left\{\frac{-1}{m^q},0,\frac{1}{m^q},\frac{2}{m^q},\ldots,1\right\}.
\end{equation}
Since $2$ is larger than the length of the interval $\left[\frac{-1}{m^q},1\right]$, for any $m$ considered, then the two conditions, \eqref{eq:lema2_2} and \eqref{eq:lema2_3}, cannot be simultaneously fulfilled for any $b_g$ and the two cases are therefore necessarily disjoint.
Now, assume that $b_i=0$ and $b_j=2/3$ and that $h=f_{i}^{-1}gf_j$  belongs to $S_{q}^{k}$.\\
By \eqref{eq:belong_s}, we obtain that the constant term of $h$ satisfies
\begin{equation}
\label{eq:lema2_4}
\frac{2}{m^q}+3b_g \in \left\{\frac{-1}{m^q},0,\frac{1}{m^q},\frac{2}{m^q},\ldots,1\right\}.
\end{equation}
Assume further that $h^{*}=f_{i}^{-1}gf_j$ belongs to $S_{q}^{k}$ with $b_i=2/3$ and $b_j=2/3$.
Again, by \eqref{eq:belong_s}, the constant term of $h^{*}$ satisfies
\begin{equation}
\label{eq:lema2_5}
\frac{2}{m^q}+3b_g-2 \in \left\{\frac{-1}{m^q},0,\frac{1}{m^q},\frac{2}{m^q},\ldots,1\right\}.
\end{equation}
Since $2$ is larger than the length of the interval $\left[\frac{-1}{m^q},1\right]$, for any possible $m$, then the two conditions, \eqref{eq:lema2_4} and \eqref{eq:lema2_5}, cannot be simultaneously fulfilled for any $b_g$ and the two cases are, again, disjoint and the Lemma is proved.
\end{proof}

\noindent Lemmas \ref{le:lemma1} and \ref{le:lemma2} allow us to caracterize the substitution matrices $M_{q}^{k}$. Each matrix $M_{q}^{k}$ is a $(0,1)$-matrix, whose spectral radius is less or equal to $2$. We will show that, in fact, the spectral radius is strictly less than $2$.
In order to do that, we will consider a matrix $N^q$. This matrix will correspond to the substitution matrix of the Digraph IFS, $D$, whose nodes are all possible functions of the form
$$ \frac{1}{m^q}x+\frac{s}{m^q}, $$
where $s\in\{-1,\ldots,m^q\}$. The Digraph IFS, $D$, has an edge from a node $g$ to a node $h$ if
\begin{equation}
\label{eq:relation}
h=f_{i}^{-1}gf_{j},
\end{equation}
which implies that the entry $(g,h)$ of the matrix $N^q$ will be different from zero.

Note that, due to relation \eqref{eq:relation}, then, relation \eqref{eq:belong_s} holds for the constant term of $h$ and Lemmas \ref{le:lemma1} and \ref{le:lemma2} apply to the matrix $N^{q}$ without any change in the respective proof. So, $N^q$ is a $(0,1)$-matrix whose row entries sum at most $2$.

Furthermore, under these assumptions, the matrices $M_{q}^{k}$ are principal submatrices of $N^q$, which means that if we are able to bound the spectral radius of $N^q$ away from $2$, uniformly on $q$, then, by \eqref{eq:raio-sub-raio}, the same will apply to $M_{q}^{k}$.

In what follows, we will use the notation $a \equiv b \mod p$, to express the fact that $a$ and $b$ are congruent modulo $p$.

\begin{lemma}
\label{le:lemma3}
Assume that $m$ in the definition of $T$, in \eqref{eq:dynamics}, is such that $m$ is not divisible by $3$, \ie $m\not\equiv 0 \mod 3$. Then, the matrix $N^{q}$ has a spectral radius less or equal to $\sqrt{3}$, \ie $$\rho(N^{q})\leq \sqrt 3.$$
\end{lemma}
\begin{proof}
Let $g$ be a function of the form
$$ \frac{1}{m^q}x+\frac{s}{m^q},  $$
for $s\in\{-1,\ldots,m^q\}$. If the entry, $(g,h)$, of $N^q$ is different from zero, then $h=f_{i}^{-1}gf_{j}$ and relation \eqref{eq:belong_s} holds, which means that the constant term of $h$ satisfies
$$3\left(\frac{1}{m^q}b_j+b_g-b_i\right)\in \left\{\frac{-1}{m^q},0,\frac{1}{m^q},\frac{2}{m^q},\ldots,1\right\}. $$
Depending on the value of $b_i$ and $b_j$, we have four different cases.

If $b_j=0$ and $b_i=2/3$, then, the constant term of $h$ satisfies $$\frac{3s-2m^{q}}{m^q}\in \left\{\frac{-1}{m^q},0,\frac{1}{m^q},\frac{2}{m^q},\ldots,1\right\}.$$

If $b_j=0$ and $b_i=0$, then, the constant term of $h$ satisfies $$\frac{3s}{m^q}\in \left\{\frac{-1}{m^q},0,\frac{1}{m^q},\frac{2}{m^q},\ldots,1\right\}.$$

If $b_j=2/3$ and $b_i=0$, then, the constant term of $h$ satisfies  $$\frac{3s+2}{m^q}\in \left\{\frac{-1}{m^q},0,\frac{1}{m^q},\frac{2}{m^q},\ldots,1\right\}.$$

If $b_j=2/3$ and $b_i=2/3$, then, the constant term of $h$ satisfies $$\frac{3s-2m^{q}+2}{m^q}\in\left\{\frac{-1}{m^q},0,\frac{1}{m^q},\frac{2}{m^q},\ldots,1\right\}.$$

Up to this point,  every entry of the matrix is indexed by functions of the form 
$ \frac{1}{m^q}x+\frac{s}{m^q}, $ with $s\in\{-1,\ldots,m^q\}$. Hence, we can associate to the entry of the matrix $(g,h)$ the index $(s,s^{*})$, where $s$ and $s^{*}$ are the numerators of the constant terms of $g$ and $h$, respectively.
An entry $(s,s^{*})$ of $N^{q}$ is nonzero if and only if $s$ is such that one the above cases is verified.

If the first case occurs, then $3s-2m^q\in \{-1,\ldots,m^q\} $. Hence, if  $N^{q}_{ss^{*}}\neq 0$, we have that  $s^{*}=3s-2m^{q}$.
If the second case occurs, then $3s\in \{-1,\ldots,m^q\}$ and if $N^{q}_{ss^{*}}\neq 0$ then $s^{*}=3s$.
For the third case, we need $3s+2$ to belong to the set $ \{-1,\ldots,m^q\} $ and if $N^{q}_{ss^{*}}\neq 0$ then $s*=3s+2$.
If the last case is verified, then  $3s-2m^{q}+2$ belongs to $ \{-1,\ldots,m^q\} $ and if $N^{q}_{ss^{*}}\neq 0$ then $s^*=s-2m^{q}+2$.

Changing the indices for the more usual set $\{1,\ldots,m^{q}+2\}$, we obtain that $N^q$ can be characterized by
\begin{equation}
\label{eq:matrix_n}
\begin{cases} N^{q}_{i,3i-2}=1 & \mbox{if } i,3i-2 \in\{1,\ldots,m^{q}+2\} \\
N^{q}_{i,3i-4}=1 & \mbox{if } i,3i-4\in\{1,\ldots,m^{q}+2\} \\
N^{q}_{i,3i-2m^{q}-4}=1 & \mbox{if } i,3i-2m^{q}-4\in\{1,\ldots,m^{q}+2\}\\
N^{q}_{i,3i-2m^{q}-2}=1 & \mbox{if } i,3i-2m^{q}-2\in\{1,\ldots,m^{q}+2\} \\
N^{q}_{i,j}=0 & \mbox{otherwise. }
\end{cases}
\end{equation}
For a more visual representation of $N^{q}$, we may write:
\setcounter{MaxMatrixCols}{20}
\begin{equation}
\label{mt:matrix_n}
N^q=
\begin{pmatrix}
1 & 0 & 0 & \ldots & \ldots & \ldots & \ldots & 0 & 0 & 0 & 0 \\
0 & 1 & 0 &  1 & 0 & \ldots & \ldots & \ldots & 0 & 0 & 0 \\
0 & 0 & 0 & 0 & 1 & 0 & 1 & 0 &\ldots & 0 & 0 \\
0 & 0 & 0 & 0 & 0 & 0 & 0 & 1 & 0 & 1 & \ldots \\
\vdots & \vdots & \vdots  & \vdots  & \vdots  & \vdots  & \vdots  & \vdots  & \vdots  & \vdots  & \vdots \\
0 & 0 & 0 & 0 & 0 & \ldots & 0 & 0 & 0 & 0 & 0 \\
\vdots & \vdots & \vdots  & \vdots  & \vdots  & \vdots  & \vdots  & \vdots  & \vdots  & \vdots  & \vdots \\
0 & 0 & 0 & 0 & 0 & \ldots & 0 & 0 & 0 & 0 & 0 \\
\vdots & \vdots & \vdots  & \vdots  & \vdots  & \vdots  & \vdots  & \vdots  & \vdots  & \vdots  & \vdots \\
\ldots & 1 & 0 & 1 & 0 & 0 & 0 & 0 & 0 & 0 & 0 \\
0 & 0 & \ldots & 0 & 1 & 0 & 1 & 0 & 0 & 0 & 0 \\
0 & 0 & 0 & \ldots & \ldots & \ldots & 0 & 1 & 0 & 1 & 0 \\
0 & 0 & 0 & 0 & \ldots & \ldots & \ldots & \ldots & 0 & 0 & 1
\end{pmatrix}.
\end{equation}
The shape of the matrix $N^{q}$ will depend on how many sequences of $(1, 0, 1)$ will fit in $m^q+1$ columns. Since $m$ is not divisible by $3$, by Fermat's Little Theorem, 
we have
$m^2\equiv 1 \mod 3 .$
Let $q=2k$, for some $k\in\N$, then
$$m^{2k}+1\equiv 2 \mod 3.$$
If $q=2k+1$, for some $k\in\N$, then we have
$m^{2k+1}+1\equiv m+1 \mod 3 .$
Hence, in this case
$$m^{2k+1}+1\equiv 0 \mod 3 \text{  or  } m^{2k+1}+1\equiv 2 \mod 3, $$
depending on whether $m\equiv 2 \mod 3$ or $m\equiv 1 \mod 3$, respectively.

This means that we have two different cases to address either $m^{q}+1\equiv 0 \mod 3$ or $m^{q}+1\equiv 2 \mod 3$.

Let $q$ be such that $m^q+1 \equiv 2 \mod 3$. Denote by $x=(x_1,x_2,\ldots,x_{m^{q}+2})$ a vector in $\R^{m^q+2}$.
Note that there is no $i$ such that $ N^{q}_{i,3i-2}=1$ and $N^{q}_{i,3i-2m^{q}-2}=1$, in conjunction, or $N^{q}_{i,3i-2}=1$ and $N^{q}_{i,3i-2m^{q}-4}=1$, together. Similarly, there is also no $i$ such that $N^{q}_{i,3i-4}=1$ and $N^{q}_{i,3i-2m^{q}-4}=1$, together, or $N^{q}_{i,3i-4}=1$ and $N^{q}_{i,3i-2m^{q}-2}=1$.
Hence, we can write
\begin{equation*}
N^{q}x=
  \begin{pmatrix}
  x_1\\
  x_2+x_4\\
  \ldots\\
  x_{2+3\alpha}+x_{4+3\alpha}\\
  x_{m^q+1}\\
  0\\
  \ldots\\
  0 \\
  x_2\\
  x_{m^q-1-3\beta}+x_{m^q+1-3\beta}\\
  \ldots\\
  x_{m^q-1}+x_{m^q+1}\\
  x_{m^q+2}
  \end{pmatrix},
\end{equation*}
where $\alpha,\beta$ are integers that satisfy $0<\alpha,\beta<(m^q+1)/3$.

Consider the sets
$$\mathcal{A}=\left\{ i\in\N:i=2+3\alpha \text{ and } 0<\alpha<(m^q+1)/3 \right\} $$
and
$$\mathcal{B}=\left\{ i\in\N:i=m^q-1-3\beta \text{ and } 0<\beta<(m^q+1)/3 \right\} .$$
Then, $\lVert N^q x\rVert^{2}$ can be written as
\begin{equation}
\label{eq:sum_1}
 x_{1}^{2}+x_{2}^{2}+\sum_{i\in\mathcal{A}}(x_i+x_{i+2})^{2}+\sum_{j\in\mathcal{B}}(x_j+x_{j+2})^{2}+x_{m^q+1}+x_{m^q+2}.
\end{equation}
For simplicity, let
$$A:= \sum_{i\in\mathcal{A}}(x_i+x_{i+2})^{2}$$
and
$$B:= \sum_{i\in\mathcal{B}}(x_j+x_{j+2})^{2}.$$
Each coordinate of the vector $x$ appears, at most, once in $A$ and once in  $B$.
Let $i^{*}$ be such that, $x_{i^*}$ appears both in $A$ and $B$ and assume that $ i^{*}\in\mathcal{A}$. Then $i^{*}=j+2$ for some $j$ in $\mathcal{B}$.
To prove this, we proceed by contradiction. Assume that $i^*=j$ for some $j\in \mathcal{B}.$
Then,
$$ 2+3\alpha=m^q-1-3\beta, $$
for some integers $\alpha,\beta.$ But, this implies that $3$ divides $m$ which is an absurd.
With a similar argument, we prove that, if $i^{*}=i+2$, for some $i\in\mathcal{A}$, then $i^*=j$ for some $j\in\mathcal{B}$.

Now, let $i^{*}$ be such that $x_{i^*}$ appears in $A$ and in $B$ and assume that $i^*\in\mathcal{A}$. Then $i^*=j+2$ for some $j\in\mathcal{B}$. We will show that $x_{i^*+2}$ does not appear more than once in \eqref{eq:sum_1}.
If $x_{i^*+2}=x_1$ or $x_{i^*+2}=x_3$, then $i^*+2=1 $ or $i^*+2=3 $  and $i^{*}$ cannot belong to $\mathcal{A}$.
On other hand, if $x_{i^*+2}=x_{m^q+2}$ then
$$ i^{*}+2=m^q+2 $$
and for some $\beta\in\N$,
$$m^q+1-3\beta=m^q+2$$
which implies that $3\beta=-1$. This is an absurd. A similar argument can be made to show that $x_{i^*+2}\neq x_{m^q}$.
To finish, assume that exists a $j^{*}\in\mathcal{B}$ such that $i^{*}+2=j^{*}$ or $i^{*}+2=j^{*}+2$. If $i^{*}+2=j^{*}$, then there exist integers $\beta$ and $\beta^{*}$ such that
$$ m^{q}+1-3\beta=m^q-1-3\beta^{*}.$$
Hence, $ 2=3(\beta-\beta^{*}) $
which is also an absurd. If $i^{*}+2=j^{*}+2$, then $j^{*}\in\mathcal{A}$. This is impossible as proved earlier.

Similarly, if $x_{i^*}$ appears in $A$ and $B$ and $i^{*}=i+2$, for some $i\in\mathcal{A}$, then $x_i$ cannot appear more than once in \eqref{eq:sum_1}.

Using Proposition \ref{prop:peter_paul}, we can write that, for any $\varepsilon>0$,
$$ (x_i+x_{i+2})^2\leq(1+\varepsilon)x_{i}^{2}+(1+1/\varepsilon)x_{i+2}^{2} .$$
As proved above and due to the matrix pattern, for $x_l$ to appear in the sum $A$ and $B$ then $l=i$ for some $i\in\mathcal{A}$ and $l=j+2$ for some in $j\in \mathcal{B}$. Hence, for all $\varepsilon>0$,
\begin{multline}
\label{eq:sum_2}
\lVert N^qx\rVert^{2}\leq x_{1}^{2}+x_{2}^{2}+\sum_{i\in\mathcal{A}}(1+\varepsilon)x_{i}^{2}+(1+1/\varepsilon)x_{i+2}^{2}+\\ \sum_{i\in\mathcal{B}}(1+1/\varepsilon)x_j^2+(1+\varepsilon)x_{j+2}^{2}+x_{m^q+1}^2+x_{m^q+2}^2
\end{multline}
and we can establish that there are coefficients $c_l$ such that
\begin{equation}
\label{eq:sum_3}
\lVert N^q x\rVert^{2}\leq\sum_{l=1}^{m^q+2}c_l x_{l}^{2}.
\end{equation}
So, choosing $\varepsilon=0.5$ and if $x_l$ appears in both sums $A$ and $B$, then $c_l=2(1+\varepsilon)=3$. Furthermore, if $x_{l^{*}}$ is another coordinate such that the term $(x_l+x_{l^*})^{2}$ appears only in $A$ or in $B$, then $c_{l^*}=(1+1/\varepsilon)\leq 3$, since $x_{l^*}$ does not appear anywhere else in \eqref{eq:sum_2}.
On other hand, if $x_1 $ appears either in $A$ or $B$ then $c_1$ is equal to $1+1+\varepsilon\leq 3$ and the same conclusion holds for $c_2$, $c_{m^q+1}$ or $c_{m^q+2}$. Hence, for every $1\leq l\leq m^{q}+2$,  we have $c_l\leq 3$ and therefore
$$ \lVert N^q x\rVert^{2}\leq3\sum_{l=1}^{m^q+2} x_{l}^{2}\leq 3 \lVert x\rVert^{2}.$$
Consequently, by \eqref{eq:raio-norma},
$$\rho(N^{q})\leq \sqrt{3} .$$
If $m^{q}+1\equiv 0 \mod 3$, in a very similar way, we obtain
$$ \lVert N^q x\rVert^{2}=x_{1}^{2}+\sum_{i\in\mathcal{A}}(x_i+x_{i+2})^{2}+\sum_{j\in\mathcal{B}}(x_j+x_{j+2})^{2}+x_{m^q+2}. $$
Then using the inequality
\begin{multline*}
\lVert N^qx\rVert^{2}\leq x_{1}^{2}+\sum_{i\in\mathcal{A}}(1+1/\varepsilon)x_{i}^{2}+(1+\varepsilon)x_{i+2}^{2}+ \sum_{j\in\mathcal{B}}(1+\varepsilon)x_j+(1+1/\varepsilon)x_{j+2}^{2}+x_{m^q+2},
\end{multline*}
with $\varepsilon=0.5$, the proof follows for all $q$.
\end{proof}

\begin{remark}
We point out that each of the Digraph IFS associated to the intersection $T^{-q}(\mathcal{C})\cap\mathcal{C}$ satisfies the open set condition. Each digraph is composed of only two different similarities, $x/3$ and $x+2/3$. Hence, choosing $\Omega_v=(0,1)$ for every $v\in V$, we can check that the conditions in Definition~\ref{de:opensetcondition} are easily satisfied.
\end{remark}
\begin{proof}[Proof of Proposition~\ref{prop:dimension}]
Recalling that the matrices $M^{k}_{q}$ of each Digraph IFS associated to the intersection $T^{-q}(\mathcal{C})\cap\mathcal{C}$ are principal submatrices of $N^q$, then \eqref{eq:raio-sub-raio} implies that
$$\rho(M_{q}^{k})\leq\rho(N^q).$$

Hence, if $m$ is not divisible by $3$, Lemma~\ref{le:lemma3} gives us $\rho(M_{q}^{k})\leq\sqrt3$. Consequently, noting that each Digraph IFS associated with a matrix $M^{k}_q$ satisfies the open set condition and is composed of only two different similarities, $x/3$ and $x+2/3$, both with ratio $1/3$, we can apply \cite[Theorem~3 and 4]{MW88} to estimate the Hausdorff dimension of $T^{-q}(\mathcal{C})\cap\mathcal{C}$, namely,
\begin{equation*}
\text{dim}_H(T^{-q}(\mathcal{C})\cap\mathcal{C})\leq\frac{\log \sqrt 3}{-\log 1/3}=\frac12,
\end{equation*}
for all $q\in\N$.
Moreover, one can check that conditions of \cite[Theorems~1.1 and 2.7]{DN04} are satisfied in our setting and therefore $\text{dim}_H(T^{-q}(\mathcal{C})\cap\mathcal{C})=\text{dim}_B(T^{-q}(\mathcal{C})\cap\mathcal{C})$, which allows us to obtain:
\begin{equation}
\label{eq:dimension}
\text{dim}_B(T^{-q}(\mathcal{C})\cap\mathcal{C})\leq\frac{\log \sqrt 3}{-\log 1/3}=\frac12<\frac{\log2}{\log3}=\text{dim}_B(\mathcal{C}).
\end{equation}

So far, $m$ is not divisible by $3$. Using the self-similarity of the Cantor set,  $\mathcal{C}$, it is possible to extend our findings to the cases where $m=3^{k}c$, for some integers $c,k>1$ such that $c$ is not divisible by $3$. Figure \ref{fig:figure_3kc} intends to illustrate our reasoning for the case where $k=1$, $c=2$ and $q=1$.

Consider the map $\tilde T(x)=cx \mod 1$. We claim that $$\text{dim}_{B}(T^{-q}(\mathcal{C})\cap \mathcal{C})=\text{dim}_{B}(\tilde T^{-q}(\mathcal{C})\cap\mathcal{C}).$$
Let $g_\gamma\colon\R\to\R$ be given by $g(x)=\gamma x$ for all $x\in\R$. 
The set $T^{-q}(\mathcal{C})\cap\mathcal{C}$ is obtained by intersecting $3^{kq}c^{q}$ copies of the set $g_{3^{-kq}c^{-q}}(\mathcal{C})$ distributed side by side along the interval $[0,1]$, with the set $\mathcal{C}$. 
\begin{figure}[h]

\begin{tikzpicture}[decoration=Cantor set, line width=1mm,scale=2.5]
\draw decorate{ (0,.27) -- (3,.27) };
\node at (3.17,.27){$\C_{1}$};
\draw (-.017,0) -- (3.019,0);
\node at (3.13,0.1) {$1$};
\node at (-0.13,0.1) {$0$};
\draw (0,0) -- (0,-.1);
\draw (0.5,0) -- (0.5,-.1);
\draw (1,0) -- (1,-.1);
\draw (1.5,0) -- (1.5,-.1);
\draw (2,0) -- (2,-.1);
\draw (2.5,0) -- (2.5, -.1);
\draw (3,0) -- (3.0, -.1);
\draw [dashed, line width=0.5mm] (3, -.05) to  (3.2,-.05);
\draw [->, dashed, line width=0.5mm] (3.2, -.05) to  (3.2,-.25);
\node at (3.2,-.38) {$T^{-1}(\C)$};
\draw [->, dashed, line width=0.5mm] (2.75, -.05) to  (2.75,-.5);
\node at (2.75,-.62){Copy of};
\node at (2.75,-.8){$g_{3^{-1}2^{-1}}(\C)$};
\draw [dashed, line width=0.5mm] (-.019,-.23) to (1.019,-.23);
\draw [dashed, line width=0.5mm]  (-.019,-.23) to  (-.019,-.15);
\draw [dashed, line width=0.5mm]  (1.019,-.23) to  (1.019,-.15);
\draw [->, dashed, line width=0.5mm]  (0.51,-.23) to  (0.51,-.5);
\node at (0.51, -.62){Copy of};
\node at (0.51, -.82){$g_{3^{-1}}(\tilde T^{-q}(\mathcal{C}))$};
\end{tikzpicture}
\caption{For $c=2$ and $k=1$ this figure illustrates the relation between $\tilde T^{-1}(\C)\cap \C$ and $T^{-1}(\C)\cap \C$. \label{fig:figure_3kc}}
\end{figure}
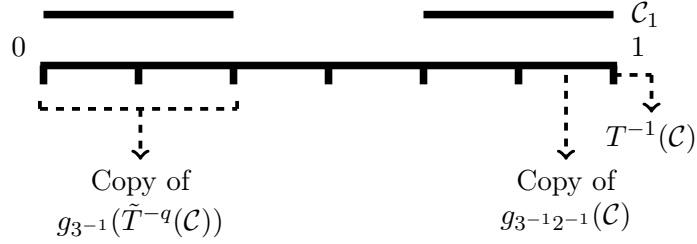

Note that because of the self-similarity of $\C$, the intersection of each of the $2^{kq}$ connected components of $\C_{kq}$ with $\C$ is a copy of $g_{3^{-kq}}(\C)$. Moreover, each of the $2^{kq}$ connected components of $\C_{kq}$ meets exactly $c^{q}$ of the copies of the set $g_{3^{-kq}c^{-q}}(\mathcal{C})$ that constitute $T^{-q}(\mathcal{C})$. Therefore, the intersection of the set $T^{-q}(\mathcal{C})\cap\mathcal{C}$ with each of the $2^{kq}$ connected components of $\C_{kq}$ is a copy of $g_{3^{-kq}}(\tilde T^{-q}(\mathcal{C})\cap\mathcal{C})$ and the claim follows.
\end{proof}

\begin{remark}
We note that when $m=3^k$, for some $k\in\N$, one can check that the matrices $N^q$ have a spectral radius equal to $2$, which means that the box dimension of $T^{-q}(\mathcal{C})\cap\mathcal{C}$ is equal to the box dimension of $\mathcal C$. For example, if $m=3$ and $q=1$ then
$$N^q=\left(
\begin{array}{ccccc}
 1 & 0 & 0 & 0 & 0 \\
 0 & 1 & 0 & 1 & 0 \\
 0 & 0 & 0 & 0 & 1 \\
 0 & 1 & 0 & 1 & 0 \\
 0 & 0 & 0 & 0 & 1 \\
\end{array}
\right),$$ which can be easily checked to have a spectral radius equal to $2$. This is consistent with what we proved in Theorem~\ref{thm:3kx_mod1}.
\end{remark}

\subsection{From dimension estimates to EI estimates}
\label{subsec:dimension-EI}

In this section we show how to make use of the information regarding the irrelevance (in terms of box dimension) of the intersection of $\mathcal C$ with its iterates, in order to compute the EI from O'Brien's formula \eqref{eq:OBriens-formula}. Essentially, we have to translate the difference between the box dimension of $\mathcal C$ and $T^{-q}(\mathcal C)\cap\mathcal C$ to the difference between the Lebesgue measure of the respective convex hull approximations of decreasing size. In order to that we will use some ideas used by Newhouse, in \cite{N79}, to study invariants of Cantor sets, such as \emph{thickness}, to prove the abundance of wild hyperbolic sets.

Again, let $\C$ denote the ternary Cantor set and $\C_n$ its $n$-th approximation consisting of $2^n$ disjoint intervals of length $3^{-n}$ and let $\CC_n$ denote the collection of intervals whose disjoint union forms $\C_n$. Note that $\C_n$ is a set while $\CC_n$ is a collection of sets.
Consider that $T^{-q}(\CC_n)$ is the collection of all the connected components of $T^{-q}(\C_n)$.
We consider the set $\C\cap T^{-q}(\C)$. Note that $\C_n\cap T^{-q}(\C_n)\downarrow \C\cap T^{-q}(\C)$. We let $\overline A$ denote the closure of $A$, $\mathring A$ its interior and $A^c$ its complement.
Define
\begin{align}
N_{3^{-n}}&=\#\{I\in\CC_n: \mathring I\cap (\C\cap T^{-q}(\C))\neq \emptyset \},\\
N_{3^{-n}}^*&=\#\{I\in\CC_n: \mathring I\cap (\C_n\cap T^{-q}(\C_n))\neq \emptyset \}.
\end{align}
Since  $\C\cap T^{-q}(\C)\subset \C_n\cap T^{-q}(\C_n)$, then $N_{3^{-n}}\leq N_{3^{-n}}^*$. However, one can prove that:
\begin{proposition}
\label{prop:N-relation}
If $n$ is sufficiently large so that $3^{-n}\leq m^{-q}$, we have $N_{3^{-n}}=N_{3^{-n}}^*$.
\end{proposition}
In order to prove the proposition, we need the following result which follows from the thickness property of the Cantor set $\C$. 
\begin{definition}
Let $\Lambda$ be a Cantor set (not necessarily the Cantor set $\C$ but homeomorphic to $\C$). To define thickness, we consider the gaps of $\Lambda$: a gap of $\Lambda$ is a connected component of $\R\setminus \Lambda$; a bounded gap is a bounded connected component of $\R\setminus \Lambda$. Let $U$ be any bounded gap and $u$ be a boundary point of $U$, so $u\in\Lambda$. Let $B$ be a bridge of $\Lambda$ at $u$, i.e. the maximal interval in $\R$ such that
\begin{itemize}
\item $u$ is a boundary point of $B$; 

\item $B$ contains no point of a gap $U'$ whose length $|U'|$ is at least the length of $U$.
\end{itemize}

The thickness of $\Lambda$ at $u$ is defined as $\ro(\Lambda,u)=|B|/|U|$. The thickness of $\Lambda$, denoted by $\ro(\Lambda)$, is the infimum over these $\ro(\Lambda,u)$ for all boundary points $u$ of bounded gaps.
\end{definition}
Since in the construction of the particular ternary Cantor set $\C$, the gaps created at the $n-th$ step have the exact same size as the connected components of $\C_n$, then the thickness of $\C$ is equal to 1. The next lemma resembles the Gap Lemma in \cite{N79,PT93}, which was stated for two Cantor sets, $\Lambda_1,\Lambda_2,$ such that $\rho(\Lambda_1)\cdot\rho(\Lambda_2)>1$, and whose conclusion was that either their intersection is nonempty or one of them is contained inside a gap of the other. Since, here, we need to consider two Cantor sets, both with thickness equal to $1$, we prove the following result which allows in particular to generalise the statement of the Gap Lemma. Note that if the maximal set $\M$ was such that $\rho(\M)>1$ then we could skip this step.
\begin{lemma}
\label{lem:gap-lemma}
Let $f,g:\R\to\R$ be two affine transformations such that $f([0,1])\cap g([0,1])\neq\emptyset$, and let $\Lambda_1=f(\C)$ and $\Lambda_2=g(\C)$. Then, either $\Lambda_1\cap\Lambda_2\neq\emptyset$ or one of them is contained inside a gap of the other (\ie $f([0,1])$ is contained inside a gap of $\Lambda_2$ or $g([0,1])$ is contained inside a gap of $\Lambda_1$).
\end{lemma}

\begin{proof}
Let us assume that neither $\Lambda_1$ nor $\Lambda_2$ are contained inside a gap of the other and assume that $\Lambda_1\cap \Lambda_2=\emptyset$, and derive a contradiction. Let us denote by $G_1$ a gap of $\Lambda_1$ and $G_2$ a gap of $\Lambda_2$. We say that $(G_1,G_2)$ form a \emph{gap pair} if $G_2$ contains exactly one boundary point of $G_1$, which also contains exactly one boundary point of $G_2$. Recall that the boundary points of $G_1$ belong to $\Lambda_1$, as the boundary points of $G_2$ must belong to $\Lambda_2$. Observe that such a gap pair must always exist because $f([0,1])\cap g([0,1])\neq\emptyset$ and otherwise the points of $\Lambda_2$ could never be removed from $f([0,1])$ in order to have that $\Lambda_1\cap\Lambda_2=\emptyset$ (having in mind that neither $f([0,1])$ nor $g([0,1])$ are contained inside a gap o $\Lambda_2$ and $\Lambda_1$, respectively).
Given such a pair we build a sequence of gap pairs $(G_1^{(i)},G_2^{(i)})_{i\in\N}$ such that either $|G_1^{(i+1)}|<|G_1^{(i)}|$ or $|G_2^{(i+1)}|<|G_2^{(i)}|$.
Let $\left(G_1^{(i)},G_2^{(i)}\right)$ be given. Let $m,p\in\N$ be the smallest integers such that $G_1^{(i)},G_2^{(i)}$ appear as gaps of $f(\C_m), g(\C_p)$, respectively. Let $C_1^{\ell}, C_1^{r}$ and $C_2^{\ell}, C_2^{r}$ be the intervals of $f(\CC_m), g(\CC_p)$, respectively, that appear to the left and right of the gaps $G_1^{(i)}$ and $G_2^{(i)}$ and share the respective border points. Note that by construction we have that  $|C_1^{\ell}|=|C_1^{r}|=|G_1^{(i)}|$ and $|G_2^{(i)}|=|C_2^{\ell}|=|C_2^{r}|$. Therefore, we must have that the right endpoint of $G_2^{(i)}$ belongs to $C_1^{r}$ or the left endpoint of $G_1^{(i)}$ belongs to $C_2^{\ell}$, or both. Let us assume \emph{w.l.o.g.} that the first case happens and denote by $T$ the right endpoint of $G_2^{(i)}$. Clearly, $T\in \Lambda_2$  and, since we are assuming that $\Lambda_1\cap\Lambda_2=\emptyset$, we have $T\notin \Lambda_1$. Hence, $T$ must fall into some gap of $\Lambda_1$ inside $C_1^{r}$, which we will denote by $G_1^{(i+1)}$. Since $|C_1^{r}|=|G_1^{(i)}|$, it follows that  $|G_1^{(i+1)}|<|G_1^{(i)}|$. In this case, we set $G_2^{(i+1)}=G_2^{(i)}$ and define $\left(G_1^{(i+1)},G_2^{(i+1)}\right)$ as the new gap pair.
It follows that $\lim_{i\to\infty}|G_1^{(i)}|=0$ or $\lim_{i\to\infty}|G_2^{(i)}|=0$, or both. Assume the first and let $y_i\in G_1^{(i)}$ and $y$ be an accumulation point of $(y_i)_{i\in\N}$. Then $y$ is also an accumulation point of the right endpoints of $G_1^{(i)}$, which belong to $\Lambda_1$, and of the left endpoints of $G_2^{(i)}$, which belong to $\Lambda_2$ and, by definition of gap pair, are all inside $G_1^{(i)}$. But since $\Lambda_1$ and $\Lambda_2$ are compact sets, then $y\in \Lambda_1\cap\Lambda_2$, which is a contradiction.
\end{proof}

\begin{proof}[Proof of Proposition~\ref{prop:N-relation}]
Observe that $T^{-q}(\C_n)$ corresponds to $m^q$ copies of $\C_n$ contracted by the factor $m^{-q}$ and placed side by side on $[0,1]$.
Let $I\in\CC_n$ be an interval such that $\mathring I\cap (\C_n\cap T^{-q}(\C_n))\neq \emptyset$. Let $J\in T^{-q}(\CC_n)$ be an interval of $T^{-q}(\C_n)$ such that $J \cap \mathring I\neq\emptyset$. Note that $\Lambda_1:=I\cap\C$ and $\Lambda_2:=J\cap T^{-q}(\C)$ are copies of the original Cantor set, due to its self similarity property. In fact, if we let $f,g$ to be affine transformations such that $f([0,1])=I$ and $g([0,1])=J$, then $I\cap\C=f(\C)$ and $J\cap T^{-q}(\C)=g(\C)$. Noting that $|J|\leq |I|$, then if $J$ is not contained in any gap of $\Lambda_1$, by Lemma~\ref{lem:gap-lemma}, it follows that $\Lambda_1\cap\Lambda_2\neq\emptyset$ and therefore $I\cap (\C\cap T^{-q}(\C))\neq\emptyset$.
If $J$ is contained in some gap of $\Lambda_1$, we consider $J_1$ to be the interval of $T^{-q}(\CC_{n-1})$ that contains $J$. If $J_1$ is not contained entirely in the same gap in which $J$ is contained, then, since by the structure of the Cantor sets we must still have that $|J_1|\leq |I|$, then by the argument above we are lead to the same conclusion that $I\cap (\C\cap T^{-q}(\C))\neq\emptyset$. If $J_1$ is still contained in the same gap of $\Lambda_1$, we define $J_2$  as the interval of $T^{-q}(\CC_{n-2})$ that contains $J$ and so on, until, eventually, we find some $k\leq n$ so that $J_k$ is not contained entirely in the same gap in which $J$ is contained and $|J_k|\leq |I|$. This is guaranteed because the size of the gap of $\Lambda_1\subset I$ is smaller than $3^{-n}\leq |J_n|$.
\end{proof}
Using the results above, we can proceed with the computation of the Extremal Index, $\theta$.

Let $\tilde{N}_{3^{-n}}$ denote the minimum number of balls of radius $3^{-n}$ to cover the set $ \mathcal{C}\cap T^{-q}(\mathcal{C}).$ By definition of box dimension, we have that
$$\lim_{n\to\infty} \frac{\log \tilde{N}_{3^{-n}}}{\log 3^n} \leq \frac{\log \sqrt{3}}{\log 3}=\frac12 .$$
Hence, there exists an $\varepsilon>0$ such that
$$ \gamma :=\frac12 +\varepsilon<\frac{\log 2}{\log 3}, $$
and, for $n$ sufficiently large,
\begin{equation}
\label{eq:N_gamma}
\tilde{N}_{3^{-n}}<e^{\gamma n\log 3} .
\end{equation}
Observe that $3\tilde{N}_{3^{-n}}$ balls of radius $3^{-n}$ are enough to cover the set $ \mathcal{C}_{n-1}\cap T^{-q}(\mathcal{C}_{n-1}).$ Since  $\mathcal{C}_n\cap T^{-q}(\mathcal{C}_n)\subseteq \mathcal{C}_{n-1} \cap T^{-q}(\mathcal{C}_{n-1})$, we have
$$ \tilde{N}_{3^{-n}}\leq N_{3^{-n}} \leq 3\tilde{N}_{3^{-n}}.$$
Applying Proposition~\ref{prop:N-relation}, we obtain that, for $n$ sufficiently large (in particular, such that $3^{-n}<m^{-q}$),
$$ \tilde{N}_{3^{-n}}\leq N_{3^{-n}}^{*} \leq 3\tilde{N}_{3^{-n}}.$$
This implies that
$$\mu(\mathcal{C}_n\cap T^{-q}(\mathcal{C}_n))=\frac{1}{3^n}N_{3^{-n}}^{*}\leq\frac{3\tilde{N}_{3^{-n}}}{3^n} .$$
Hence, by \ref{eq:N_gamma},
\begin{equation}
\label{eq:estimate-1}
 \mu(\mathcal{C}_n\cap T^{-q}(\mathcal{C}_n))\leq \frac{3e^{\gamma n\log 3}}{3^n}.
 \end{equation}
Recall that the sequence of thresholds $(u_n)_{n\in\N}$ is such that $u_{n}=n$, which means that $U_n=\mathcal{C}_{n}$ and then O'Brien's formula \eqref{eq:OBriens-formula} gives:
\begin{equation}
\label{eq:estimate-OBrien}
\lim_{n\to \infty}1-\theta_n =\lim_{n\to \infty}\frac{\mu(\mathcal{C}_n\setminus \AA_{q_n,n})}{\mu(\mathcal C_n)}.
\end{equation} The set $\AA_{q_n,n}$ depends on a sequence $(q_n)_{n\in\N}$ that we are going to choose in the following way:
\begin{equation}
\label{eq:qn}
q_n:=\left\lceil n\frac{\log3}{\log m}\right\rceil.
\end{equation}
Note that $q_n=o(w_n)$, as required, and, moreover, we have $3^{-n}\leq m^{-q_n}$, for all $n\in\N$, which guarantees that we can apply Proposition~\ref{prop:N-relation} and estimate \eqref{eq:estimate-1} holds, for all $q\leq q_n$.
Then, observing that $\mathcal{C}_n\setminus \AA_{q_n,n} \subseteq \bigcup_{i=1}^{q_n}(\mathcal{C}_n\cap T^{-i}(\mathcal{C}_n))$, we get
\begin{align*}
\mu\left(\mathcal{C}_n\setminus \AA_{q_n,n} \right)&=\mu\left(\bigcup_{q=1}^{q_n}(\mathcal{C}_n\cap T^{-q}(\mathcal{C}_n))\right)\leq \sum_{q=1}^{q_n}\mu(\mathcal{C}_n\cap T^{-q}(\mathcal{C}_n))\leq 3q_n\frac{e^{\gamma n\log 3}}{3^n}.
\end{align*}

Picking up on \eqref{eq:estimate-OBrien}, we finally obtain
\begin{align*}
\lim_{n\to \infty}1-\theta_n &=\lim_{n\to \infty}\frac{\mu(\mathcal{C}_n\setminus \AA_{q_n,n})}{\mu(\mathcal C_n)}  \leq \lim_{n\to \infty}3q_n\frac{\frac{e^{\gamma n\log 3}}{3^n}}{(\frac23)^n}\leq3\lim_{n\to \infty} q_ne^{n(\log 2 - \log 3) }=0.
\end{align*}
Therefore,  $\theta=1$.

\subsection{Verification of conditions $\D_{q_n}(u_n,w_n)$ and $\D'_{q_n}(u_n,w_n)$}
\label{subsec:D-D'}
We recall that the system $([0,1],T,\mu)$ has exponential decay of correlations of BV observables against $L^1$ observables, \ie, for all $\phi\in BV$ and $\psi\in L^1(\mu)$, there exist $C>0$ and $r=\frac1{m}$ such that
$$\cv_\mu(\phi,\psi,n)\leq C r^n.
$$
The $BV$ norm of $\I_{\AA_{q_n,n}}$ is directly related with the number of connected components of $\AA_{q_n,n}$, which we need to control. In order to do that, we start by estimating, for each $q=1,\ldots,q_n$, the number of intervals of $T^{-q}(\C_n^c)$ that intersect a single connected component of $\C_n$, which we will denote by $I$. 

Recall that our choice for $q_n$ made in \eqref{eq:qn} guarantees that $|I|=3^{-n}\leq m^{-q}$, for all $q\leq q_n$, which means that the interval $I$ from $\C_n$ can intersect at most 2 of the $m^q$ copies of $\C_n$ that were contracted to fit on equally sized intervals of length $m^{-q}$ which form the set $T^{-q}(\C_n)$. We also note that $\C_n$ is built in a symmetrical way by choosing $2^n$ intervals of equal length, $3^{-n}$, which alternate with $2^n-1$ holes of different sizes. This means that the number of holes of $\C_n$ is just about its number of connected components. In order to estimate the maximum number of connected components of $T^{-q}(\C_n)$ (with length $m^{-q}3^{-n}$), which intersect the interval $I$, we define $\kappa\in\N$ such that
$$
3^{\kappa-1} m^{-q}\leq 1\leq 3^{\kappa} m^{-q},
$$ 
\ie we take $\kappa=\left\lceil q\frac{\log m}{\log3}\right\rceil$. As represented by Figure \ref{fig:components}, the structure of the Cantor set $\C$ dictates that the maximum number of components of size $m^{-q}3^{-n}$ of one of the $m^q$ copies of $\C_n$ that compose $T^{-q}(\C_n)$ and fit into the interval $I$ is at most $2^\kappa$. 

\begin{figure}[h]
\begin{tikzpicture}[decoration=Cantor set,line width=1mm,scale=2.5]
\draw (0,0) -- (3,0);
\node at (1.5,.2) {$I$};
\draw decorate{ (0,-.3) -- (.5,-.3) };
\node at (-0.5,-.3) {$i=1$};
\draw decorate{ decorate{ (0,-.6) -- (1.5,-.6) }};
\node at (-0.5,-.6) {$i=2$};
\draw decorate{ decorate{ decorate{ (0,-.9) -- (4.5,-.9) }}};
\node at (-0.5,-.9) {$i=3=\kappa$};
\end{tikzpicture}
\caption{The impact of the structure of $\C$ on the maximum number of connected components of $T^{-q}(\C_n)$ that fit into each interval $I$. \label{fig:components}}
\end{figure}
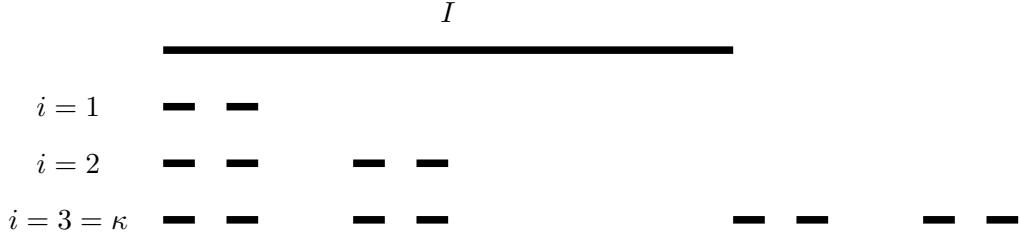

As seen above, the number of holes of one of the $m^q$ copies of $\C_n$  (or connected components of  $T^{-q}(\C_n^c)$) that fit into the interval $I$ is bounded above also by $2^\kappa$.  Since there are at most 2 of the $m^q$ copies of $\C_n$ that form   $T^{-q}(\C_n)$ which intersect $I$, then the maximum number of connected components of $T^{-q}(\C_n^c)$ that intersect $I$ is $2^{\kappa+1}=2^{\left\lceil q\frac{\log m}{\log3}\right\rceil+1}$. 

Observing that the intersection of a collection of $i$ subintervals of $I$ with another collection of $j$ subintervals of $I$ produces at most $i+j$ connected components, then $\C_n$ is formed by $2^n$ intervals like $I$. Having also in mind the choice of $q_n$ in \eqref{eq:qn}, then the number of connected components of $\AA_{q_n,n}=\C_n\cap T^{-1}(\C_n^c)\cap\ldots\cap T^{-q_n}(\C_n^c)$ is bounded above by
\begin{align*}
2^n\sum_{q=1}^{q_n}2^{\left\lceil q\frac{\log m}{\log3}\right\rceil+1}&=2^{n+2}\sum_{q=1}^{q_n}2^{\left\lfloor q\frac{\log m}{\log3}\right\rfloor}\leq 2^{n+2}\sum_{q=1}^{q_n}m^{q\frac{\log2}{\log3}}\leq 2^{n+3}m^{q_n\frac{\log2}{\log3}}\\
&\leq 2^{n+3}m^{(n\frac{\log3}{\log m}+1)\frac{\log2}{\log3}}\leq 8m 4^n.
\end{align*}
This implies that
\begin{equation*}
\Vert \I_{\AA_{q_n,n}}\ \Vert_{BV} \leq 16m \,4^n+1\leq 32m \,4^n
\label{eq:components}
\end{equation*}
Choosing, for example, $t_n=n^2$, then $t_n=o(w_n)$ and
\begin{equation*}
  \lim_{n\to\infty}\left\| \I_{\AA_{q_n,n}}\right\|_{BV}   w_{n} r^{t_n}
  \leq\lim_{n\to\infty}  \left\lfloor\tau \left(\frac32\right)^{n}\right\rfloor 32m\, 4^n \,m^{-n^2}=0
\end{equation*}
and condition $\D_{q_n}(u_n,w_{n})$ holds.

Observe that the choice of $q_n$ implies that for $q\geq q_n>n\frac{\log3}{\log m}$ we have $m^{-q}<3^{-n}$.
Recall that $T^{-q}(\C_n)$ corresponds to $m^q$ copies of $\C_n$ contracted by the factor $m^{-q}$ and placed side by side on $[0,1]$ and, since $\mu(T^{-q}(\C_n))=\mu(\mathcal C_n)=(2/3)^n$, then each such copy has a measure equal to $m^{-q}\mu(\mathcal C_n)=m^{-q}(2/3)^n$. We point out that each of the $2^n$ connected components of $\C_n$ intersects at most $\left\lfloor\frac{3^{-n}}{m^{-q}}\right\rfloor+2$ intervals of size $m^{-q}$. Hence,
\begin{align}
\mu(\mathcal{C}_n\cap T^{-q}(\mathcal{C}_n))&\leq m^{-q}\left(\frac23\right)^n\left(\left\lfloor\frac{3^{-n}}{m^{-q}}\right\rfloor+2\right)2^n\leq m^{-q}\left(\frac23\right)^n\left(\frac{3^{-n}}{m^{-q}}+2\right)2^n\nonumber\\
&\leq \left(\frac23\right)^{2n}+2\left(\frac23\right)^nm^{-q}2^n\leq 3\left(\frac23\right)^{2n}\label{eq:estimate-2}.
\end{align}
We choose $k_n=n$. Note that $k_n\xrightarrow[n\to\infty]{}\infty$ and $k_nt_n=n^3=o(w_n)$. Now, observing that $\AA_{q_n,n}\subset\C_n$, then \eqref{eq:estimate-2} implies that:
\begin{align*}
\,w_{n}\sum_{j=q_n+1}^{\lfloor w_{n}/k_n\rfloor-1}&\mu\left(  \AA_{q_n,n}\cap T^{-j}\left( \AA_{q_n,n}\right)\right)\leq \,w_{n}\sum_{j=q_n+1}^{\lfloor w_{n}/k_n\rfloor-1}\mu\left(  \C_{n}\cap T^{-j}\left( \C_{n}\right)\right)\\
&\leq \,w_{n}\sum_{j=q_n+1}^{\lfloor w_{n}/k_n\rfloor-1} 3\left(\frac23\right)^{2n}\leq 3\frac{w_{n}^2}{k_n}\left(\frac23\right)^{2n}\leq \frac3{k_n}\tau^2\left(\frac32\right)^{2n}\left(\frac23\right)^{2n}=\frac{3\tau^2}{k_n}\xrightarrow[n\to\infty]{}0
\end{align*}
and  therefore $\D'_{q_n}(u_n,w_{n})$ also holds.
Since we have already proved that  $\theta=1$, by Theorem~\ref{th:exists_evl}, we conclude the claim of Theorem~\ref{thm:2x_mod1}, \ie
\begin{equation}
\label{eq:evl_2x}
\lim_{n\to\infty} \mu(M_{w_{n}}\leq n)=\e^{ -\tau},
\end{equation}
when $m\in\N$ is not a power of $3$.

\section{The Extremal Index as a geometric indicator of the compatibility between a fractal set and a given dynamical system}
\label{sec:EI-indicator}

In order to illustrate the viability of using the EI as an indicator between the compatibility of the fractal structure of a set and a certain dynamics, we performed several numerical simulations using different dynamical systems and fractal sets. We began by testing numerically the theoretical results stated in Section~\ref{sec:statement-of-results}. Then, we kept the same maximal set and tested several different uniformly expanding and non-uniformly expanding dynamical systems and even  irrational rotations. Finally, we considered a different maximal set, which consisted on a dynamically defined Cantor set obtained from a quadratic map, and tested it against both linear dynamics (which should be incompatible) and systems compatible with the one that generated the Cantor set.

We remark that in some cases (such as with irrational rotations), the systems are outside the scope of application of the theory considered earlier. In other cases, with some adjustments to the arguments, one could actually check that conditions $\D$ and $\D'$ hold. 

We will use the EI estimator introduced by Hsing in \cite{H93a}. Namely, we will consider:
\begin{equation}
\label{eq:EIestimator}
\hat \theta_n (u,q)=\frac{ \sum_{i=0}^{n-1}\I_{T^{-i}(\AA_q(u))}}{\sum_{i=0}^{n-1}\I_{T^{-i}(U(u))}},
\end{equation}
where the sets $U(u)$ and $\AA_q(u)$ are defined in \eqref{eq:U-A-def}. The parameters $u$ and $q$ are tuning parameters which determine the quality of the estimate. In principle, one should consider high values of $u$ so that the tail behaviour is captured by the quantities in $\hat \theta_n (u,q)$, but if $u$ is too high there may not be enough information to estimate accurately the EI. Since when $\D'_{q*}(u_n,w_n)$ holds for some fixed $q*\in\N$ then $\D'_{q}(u_n,w_n)$ holds for all $q>q*$, then the parameter $q$ should not affect as much the quality of the estimator. We will test several values of $u$ and a few for $q$ and then we analyse the data to identify regions of stability of the estimator.

\subsection{The ternary Cantor set and linear dynamics}
\label{susec:simulations-1}
We illustrate numerically the existence of an EI equal to $1$ when $m$ is not a power of $3$, as stated in Theorem~\ref{thm:2x_mod1}, and the validity of the formula for the EI stated in Theorem~\ref{thm:3kx_mod1}, when $m=3^k$ for some $k\in\N$, in which case the Cantor set is invariant by the dynamics. 

The numerical simulations performed consisted in randomly generating $\ell$ uniformly distributed points on $[0,1]$ (recall that Lebesgue measure is invariant for the linear maps considered in Theorems~\ref{thm:2x_mod1} and \ref{thm:3kx_mod1}) and, for each one, compute the first $n$ iterates of the respective orbit and evaluate the observable function $\varphi$, defined in \eqref{eq:Cantor-ladder}, along each orbit. Then, for each the $\ell$ time series obtained as described above, we compute $\hat \theta_n (u,q)$,  for several values of $u$ and $q$, which are adequately chosen for the range of $u$ values.

\begin{figure}
\includegraphics[width=7cm]{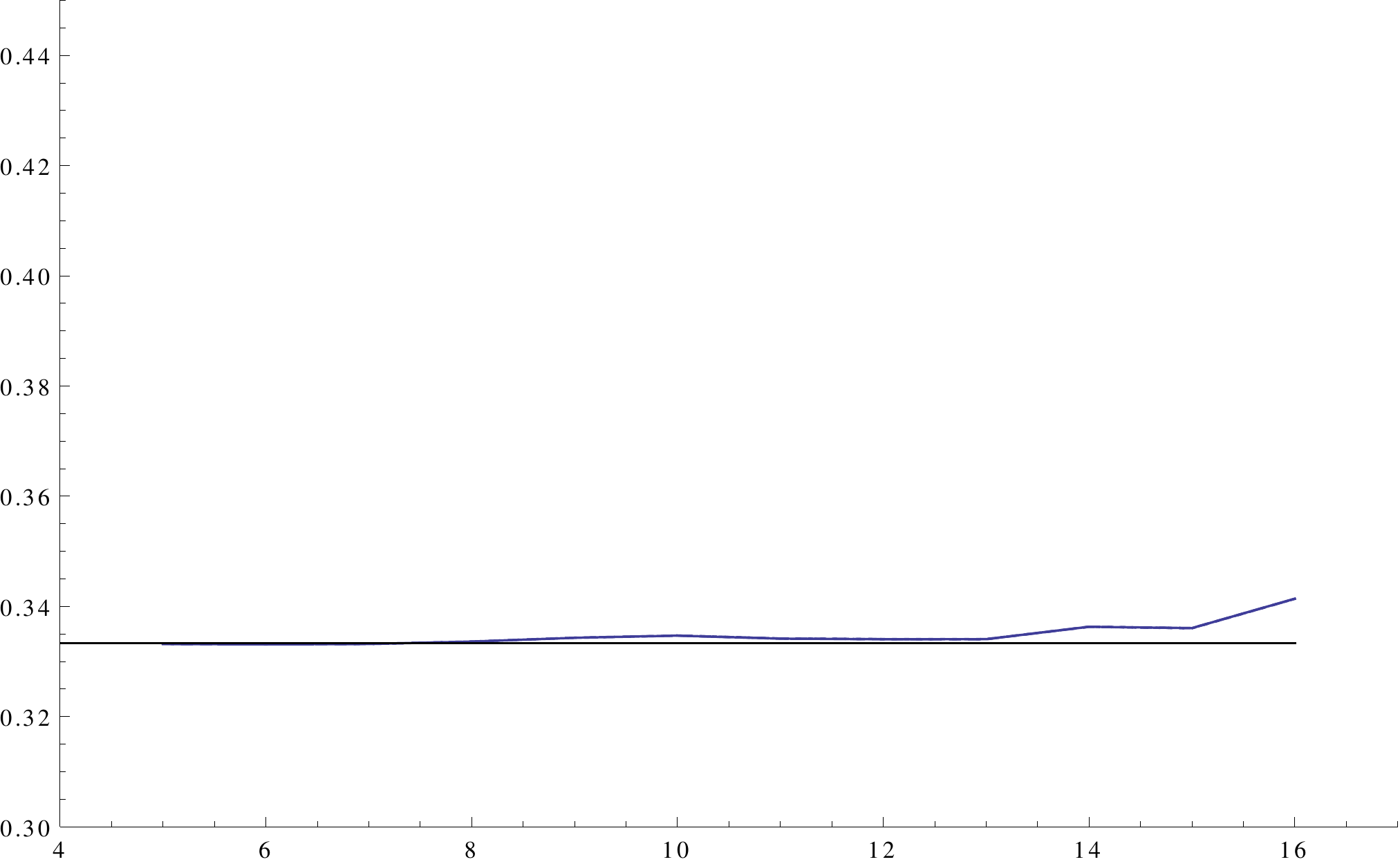} \includegraphics[width=7cm]{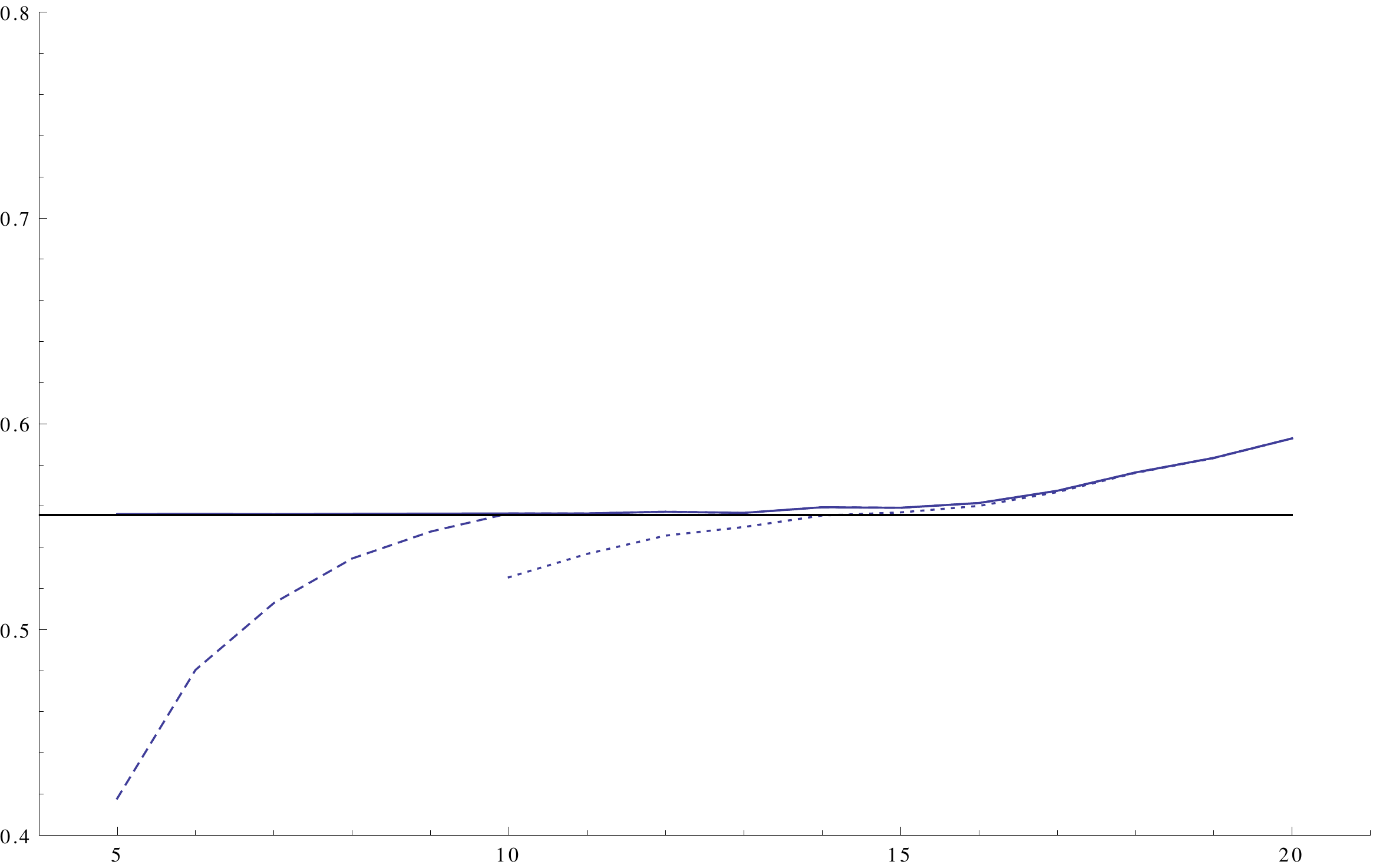} 
\caption{On the $y$-axis, mean values of $\hat \theta_n (u,q)$ for each $u$ of the $x$-axis, with $n=50.000$ and $\ell=500$.  The full line corresponds to $q=1$, the dashed line to $q=5$ and the dotted line to $q=10$. The black horizontal line represents the exact value of the EI given by Theorem~\ref{thm:3kx_mod1}. On the left, we have $T(x)=3x \mod 1$ and, on the right, $T(x)=9x \mod 1$.}
\label{fig:3xmod1} 
\end{figure}

We observe an excellent agreement between  the theoretical value of $\theta$ and the observed estimates of $\hat \theta_n (u,q)$, in the regions of stability which correspond to the values of $u$ in $[5,15]$, in the case $m=3$, and 
$[10,15]$, in the case $m=9$.

In the case $m=5$, there is also an excellent agreement between the theoretical value $\theta=1$ and the observed estimates of $\hat \theta_n (u,q)$, in the regions of stability which correspond to higher values of $u$, namely, for $u\in[15,28]$. We note that the agreement improves considerably when we increase the number iterations, $n$, which allows to have more information on the tails.

\begin{figure}[h]
\includegraphics[width=7cm]{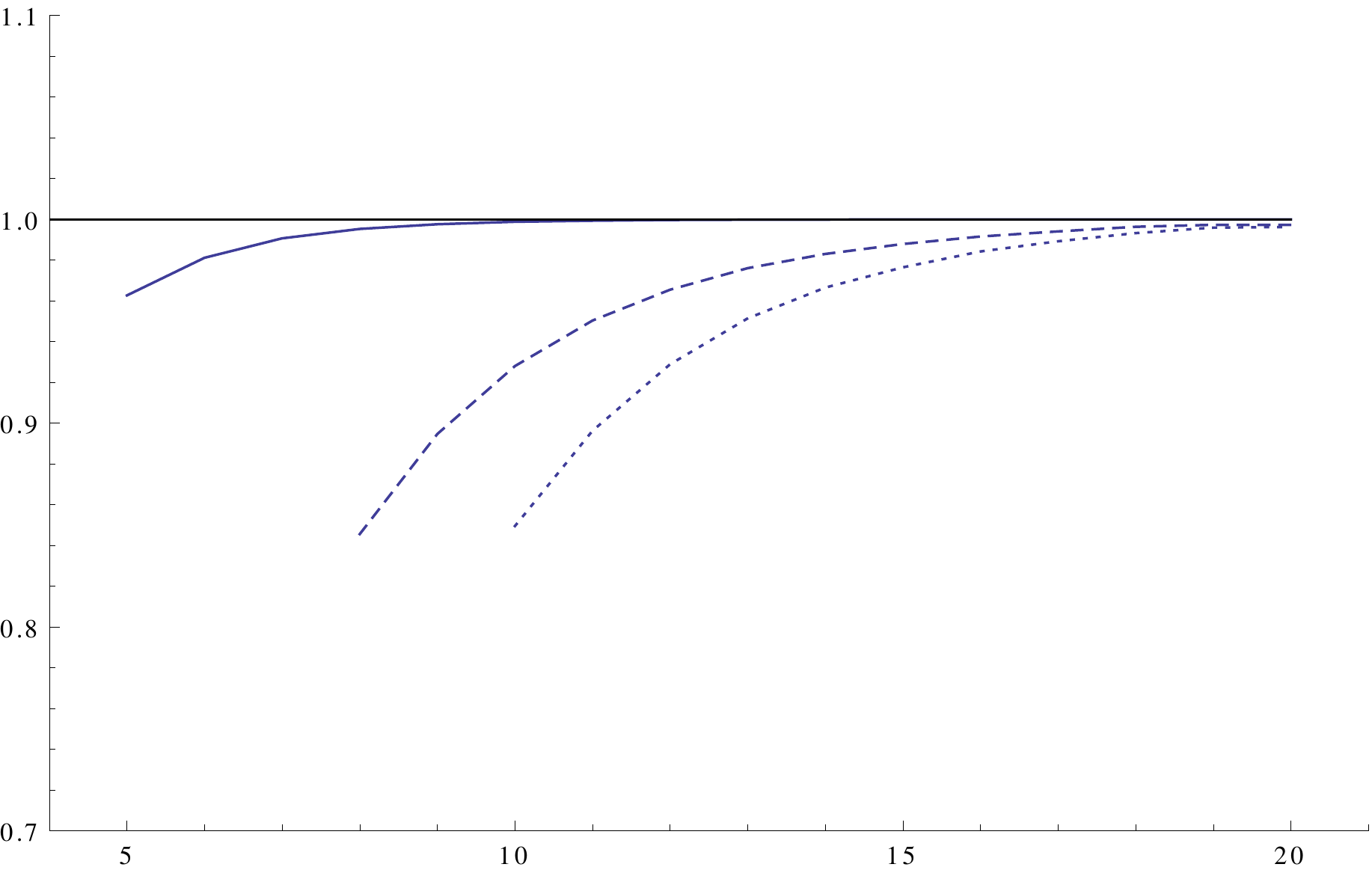} \includegraphics[width=7cm]{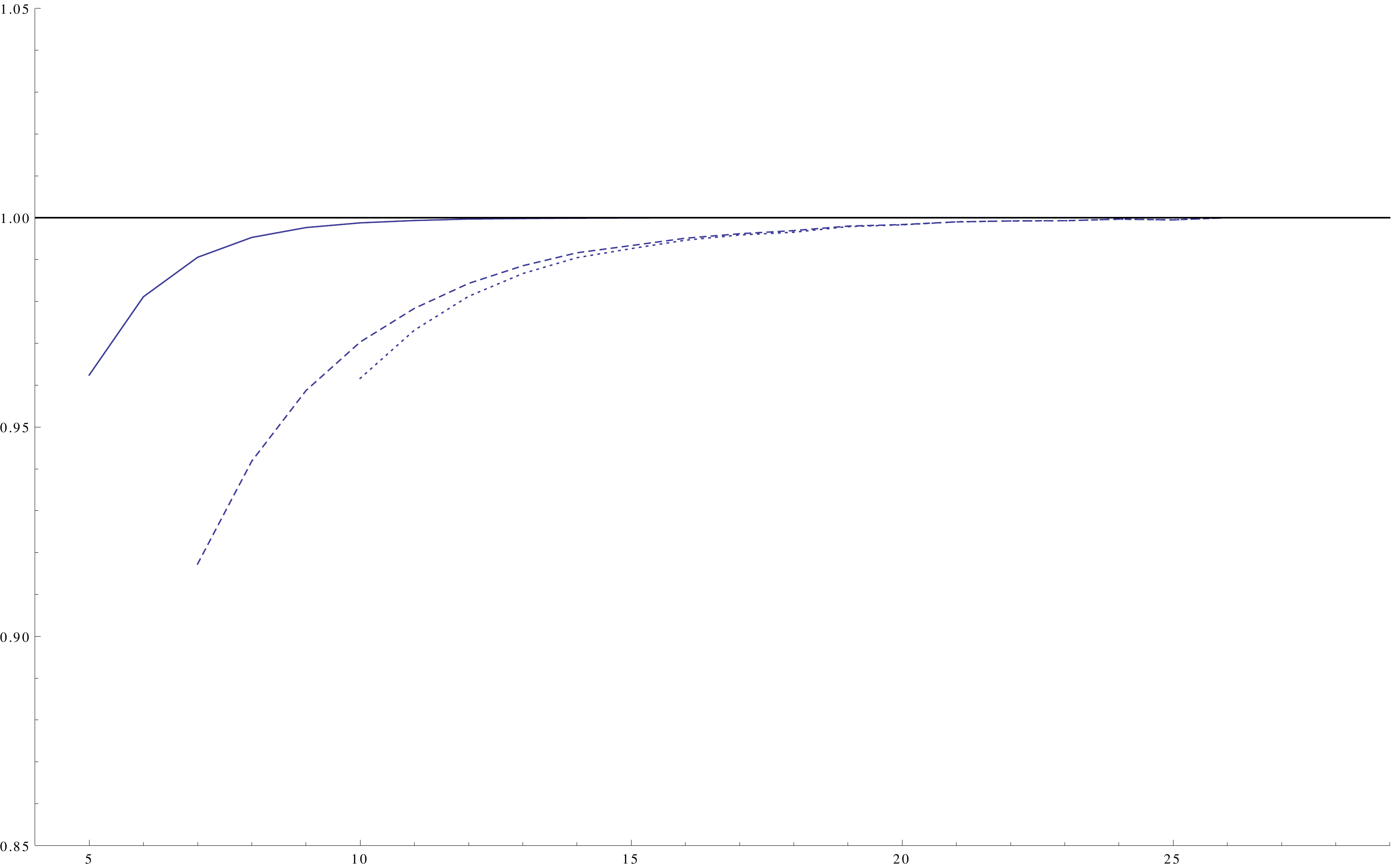} 
\caption{On the $y$-axis, mean values of $\hat \theta_n (u,q)$ for each $5\leq u\leq 20$ of the $x$-axis, with $n=50.000$. The full line corresponds to $q=1$, the dashed line to $q=5$ and the dotted line to $q=10$. The black horizontal line represents the exact value of the EI given by Theorem~\ref{thm:2x_mod1}. The dynamics is $T(x)=5x \mod 1$. On the left, $n=50.000$ and $\ell=500$. On the right, $n=500.000$ and $\ell=100$. }
\label{fig:5xmod1} 
\end{figure}
The simulations results show an excellent performance of the EI in order to distinguish between the compatibility and incompatibility of the dynamics with the structure of the Cantor set.

In the previous cases, either $T(\C)=\C$ or $T(\C)\cap\C$ is negligible. We consider a case where we have a relevant intersection $T(\C)\cap \C$, although $T(\C)\neq\C$. The idea is to consider a map that maps the left side component of $\C$ onto $\C$, while the right side component is sent to a set with a negligible intersection with $\C$. Let $T:[0,1]\to[0,1]$ be the linear map whose first branch coincides with the first branch of $3x\mod1$ and the others send each of the $5$ equally lengthed subintervals of $[2/3,1]$ onto $[0,1]$. See Figure~\ref{fig:mixedlinear-map}.
\begin{figure}[h]
\includegraphics[width=5cm]{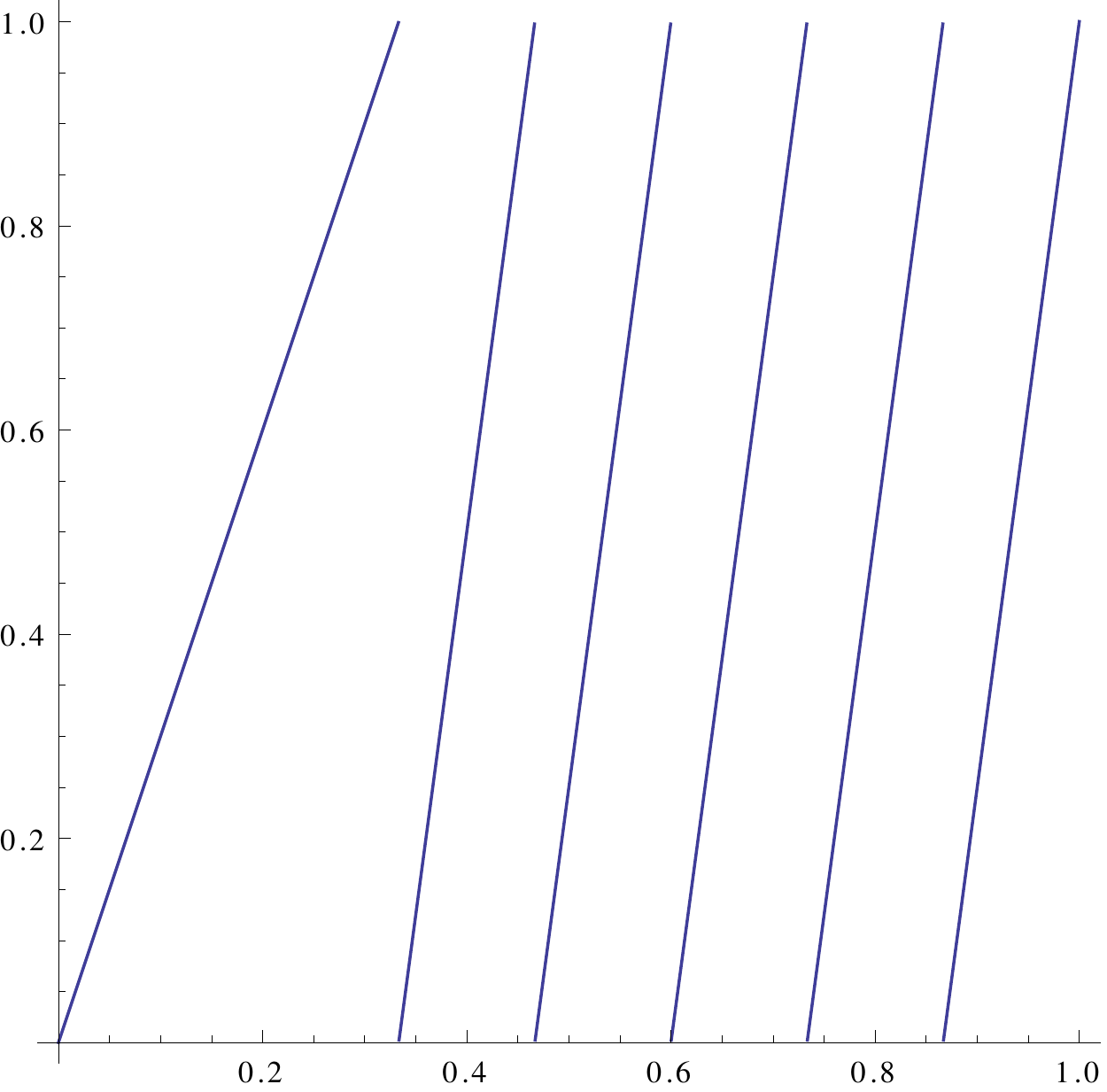}
\caption{Mixed linear map}
\label{fig:mixedlinear-map}
\end{figure}

Although this map was not considered in the previous sections, it is easy to adjust the arguments to show that an EVL applies with an EI, which is the mean between $1/3$ (the contribution from the left side) and $1$ (the contribution from the right side). Namely, using the estimates in Section~\ref{subsec:dimension-EI}, one can show that $\lim_{n\to\infty}\frac{\mu(\AA_{q_n,n}\cap[2/3,1])}{\mu(\C_n\cap[2/3,1])}=1$ and, as in Section~\ref{se:3x_mod1}, one can show that $\AA_{q_n,n}\cap[0,1/3]=(\C_n\setminus \C_{n+1})\cap[0,1/3]$, which imply:
\begin{equation}
\label{eq:EI-Mixed-linear}
\theta=\lim_{n\to\infty}\frac{\mu(\AA_{q_n,n})}{\mu(U_n)}=\lim_{n\to\infty}\frac{\mu((\C_n\setminus \C_{n+1})\cap[0,1/3])+\mu(\C_n\cap[2/3,1])}{\mu(\C_n)}=\frac12\cdot\frac13+\frac12\cdot 1=\frac23.
\end{equation}

\begin{figure}
\includegraphics[width=7cm]{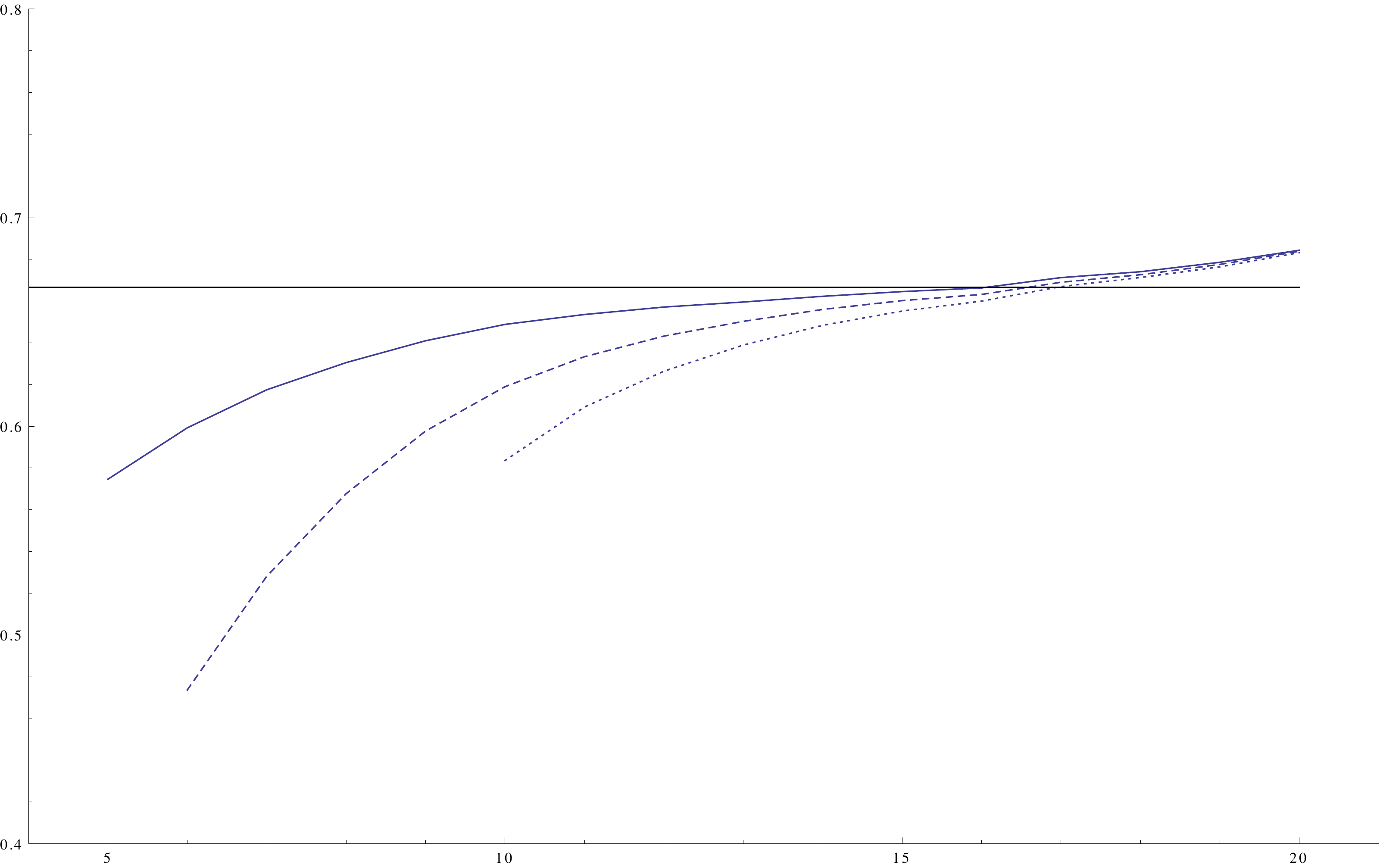} \includegraphics[width=7cm]{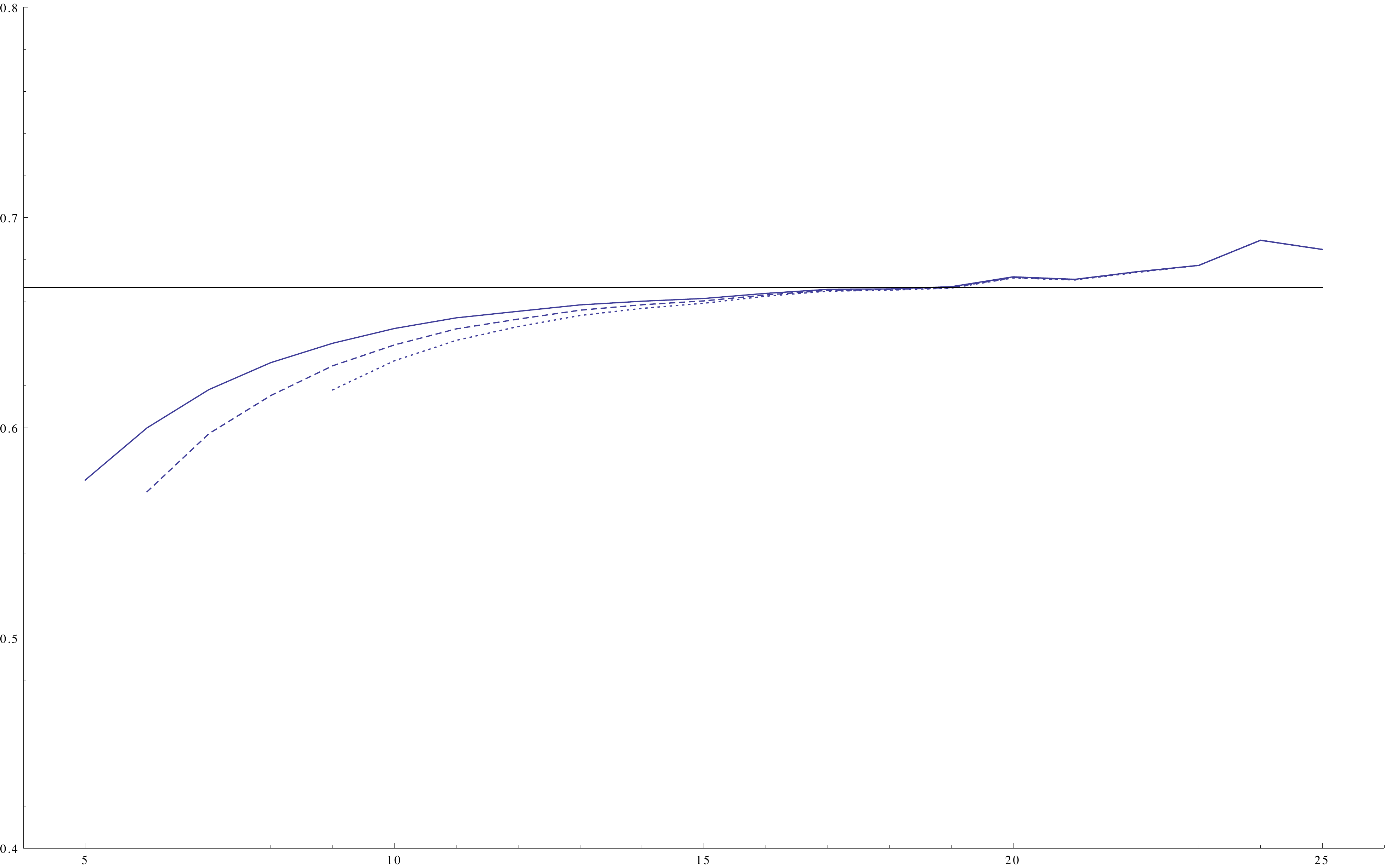} 
\caption{On the $y$-axis, mean values of $\hat \theta_n (u,q)$ for each $ u$ of the $x$-axis. The full line corresponds to $q=1$, the dashed line to $q=5$ and the dotted line to $q=10$. The black horizontal line represents the exact value of the EI given in \eqref{eq:EI-Mixed-linear}. The dynamics is described in Figure~\ref{fig:mixedlinear-map}. On the left, $n=50.000$ and $\ell=500$. On the right, $n=500.000$ and $\ell=100$. }
\label{fig:mixedlinear} 
\end{figure}

As it can be seen in Figure~\ref{fig:mixedlinear}, the numerical estimates for the EI point to the theoretical value $\theta=2/3$ and the performance of the EI estimator improves when $n$ is increased, as expected.

\subsection{The ternary Cantor set, nonlinear dynamics and irrational rotations}

We considered two different nonlinear dynamics and an ergodic rotation. The first is a uniformly expanding map resemblant to the doubling map but in which the branches are convex curves. Namely, we let
\begin{align}
\label{eq:nonlinear1}
T\colon& [0,1]\longrightarrow[0,1]\nonumber\\
&\phantomarrow{[0,1]}{x} 
\left\{\begin{array}{cc}
 \frac{4}{3} x (x+1) & 0\leq x<\frac{1}{2} \\
 \frac{4}{3} \left(x-\frac{1}{2}\right) \left(x+\frac{1}{2}\right) & \frac{1}{2}\leq x\leq 1 \\
\end{array}\right.
\end{align}
This map does not seem to have any compatibility with the ternary Cantor set and in fact the simulation results illustrate an EI estimate equal to 1, which is observed for high values of $u\approx 20$. (See top left panel in Figure~\ref{fig:nonlinear}). We note that, for this particular map, some adjustments to the arguments presented in \ref{subsec:D-D'} would allow to check conditions $\D(u_n,w_n)$ and $\D'(u_n,w_n)$. However, the box dimension estimates used in Section~\ref{subsec:box-dimension-estimates} cannot be easily adapted and therefore we cannot state that the EI is indeed $1$, despite the numerical evidence. 

\begin{figure}
\includegraphics[width=7cm]{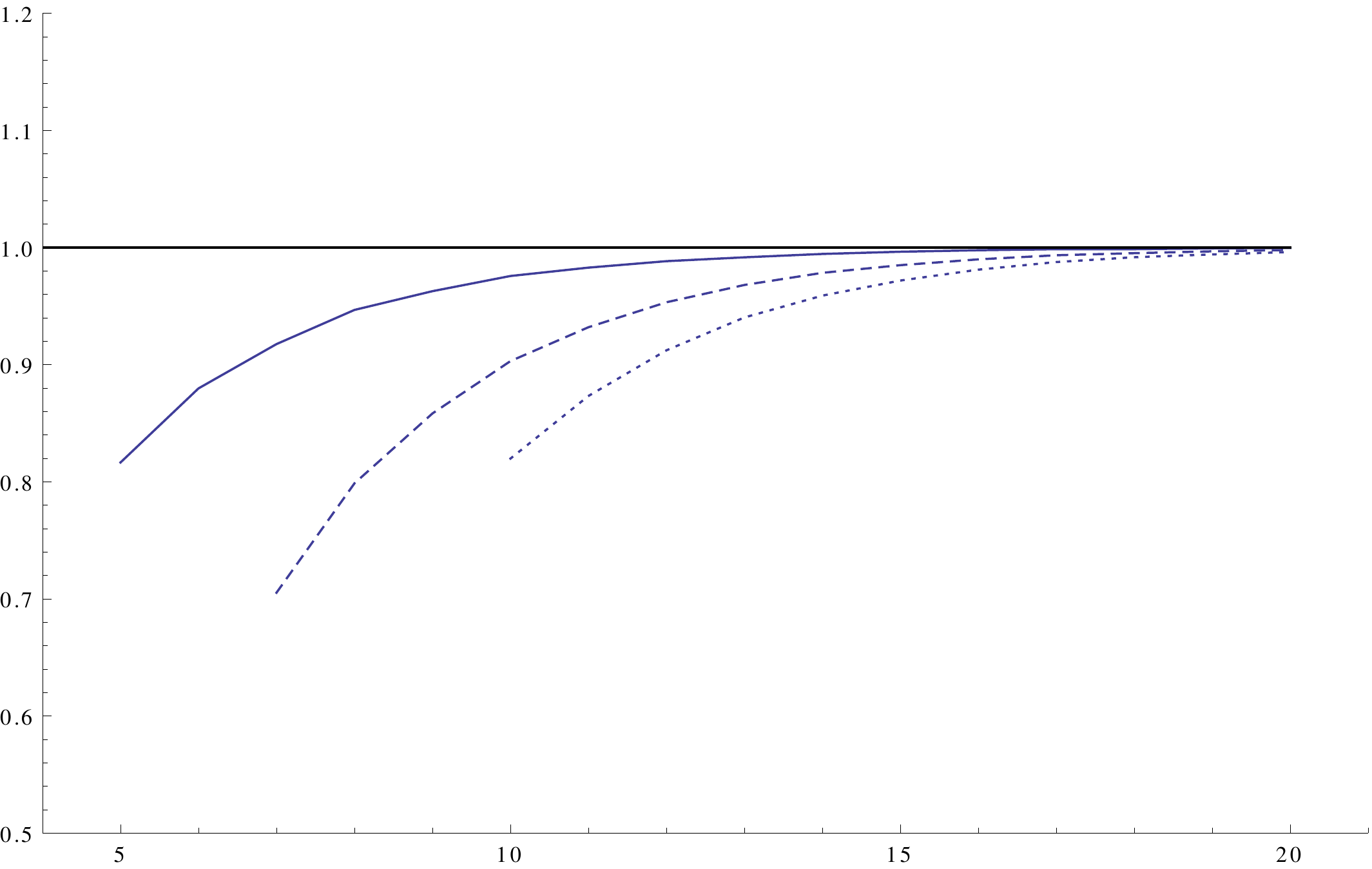} \includegraphics[width=7cm]{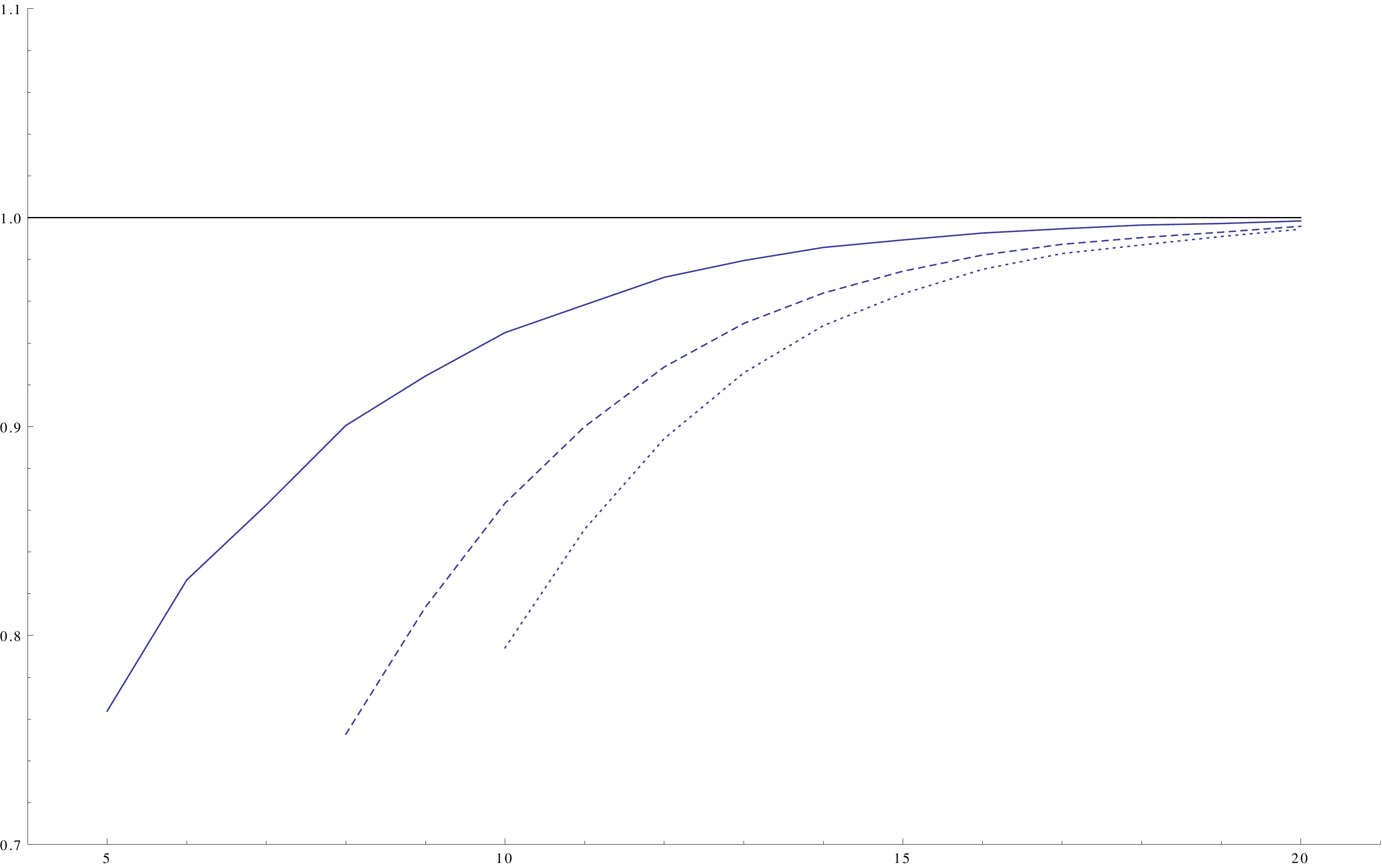} \\
\includegraphics[width=7cm]{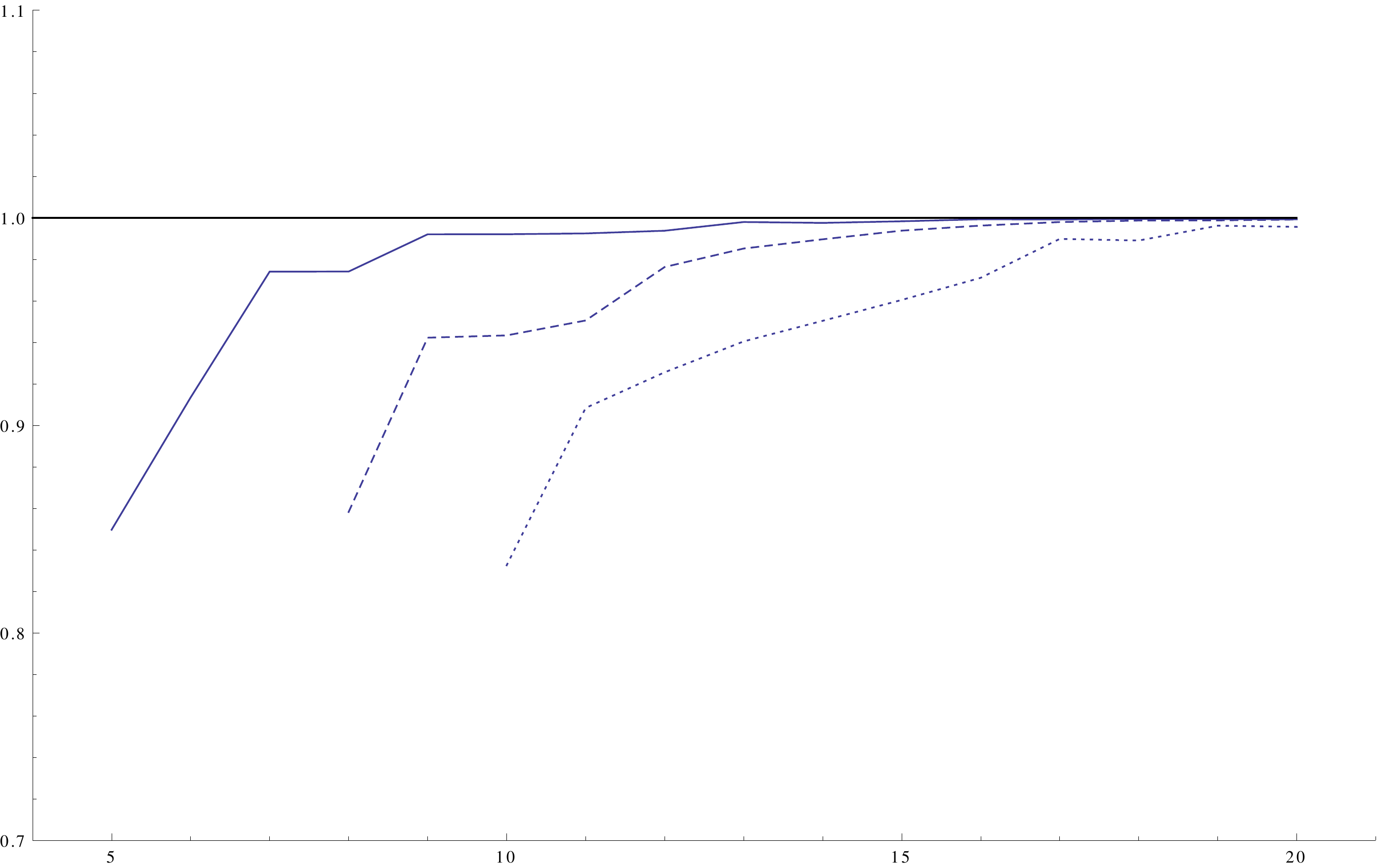}
\caption{On the $y$-axis, mean values of $\hat \theta_n (u,q)$ for each $5\leq u\leq 20$ of the $x$-axis, with $n=50.000$ and $\ell=500$. The full line corresponds to $q=1$, the dashed line to $q=5$ and the dotted line to $q=10$. The black horizontal line represents the expected value for the EI. On the top left $T$ is given by \eqref{eq:nonlinear1}, on the top right $T$ is given by \eqref{eq:nonlinear2} and on the bottom $T(x)=x+\pi/3\mod1$. }
\label{fig:nonlinear} 
\end{figure}

Then, we also considered the Gauss map, which is a non-uniformly expanding map, but still with very good mixing properties. 
\begin{equation}
\label{eq:nonlinear2}
\begin{array}{cc}
T\colon& [0,1]\longrightarrow[0,1]\\
&\phantomarrow{[0,1]}{x} 
\frac1x-\left\lfloor\frac1x\right\rfloor
\end{array}
\end{equation}
We remark that for this map is not possible to adapt easily the arguments used in \ref{subsec:D-D'} in order to check conditions $\D(u_n,w_n)$ and $\D'(u_n,w_n)$, since it has countably many branches, which makes the estimates for the number of connected components of $\AA_{q_n,n}$, obtained earlier, useless. Nonetheless, the numerical simulations also reveal that, on the region of stability of the estimator (for high values of $u$), one gets an EI equal to 1, which indicates that the dynamics is incompatible with the structure of the ternary Cantor set $\C$ (See top right panel in Figure~\ref{fig:nonlinear}).

Finally, we also considered an irrational rotation $T:[0,1]\to[0,1]$ given by $T(x)=x+\frac\pi3\mod1$, as in \cite{MP16}, and, in coherence with the numerical simulations performed there, we also obtain a numerical evidence that the EI is $1$. We remark that irrational rotations are completely outside the scope of application of the theoretical results obtained here, which depend heavily on the exponential decay of correlations of the systems considered.

\subsection{A different Cantor set}

We also considered for a maximal set a dynamically defined Cantor set  as described in Section~\ref{subsec:dynamic-Cantor}. Namely, in this case $\Lambda$ is generated by the quadratic dynamical system $g:\R\to\R$ such that $g(x)=6x(1-x)$, \ie 
$$\Lambda=\{x\in[0,1]\colon\; g^{n}(x)\in[0,1] \;\text{for all $n\in\N$}\}.
$$
In this case we define the observable map
\begin{align}
\label{eq:observable-diff-Cantor}
\varphi\colon& [0,1]\longrightarrow[0,1]\nonumber\\
&\phantomarrow{[0,1]}{x}\left\{ \begin{array}{ll}
    n, & \hbox{ if }n=\inf\{j\in\N\colon\; g^j(x)\notin [0,1]\} \\
    \infty, & \hbox{$x\in\Lambda$}
  \end{array}\right..
\end{align}
We studied numerically the behaviour of two systems. The first one is defined by
\begin{align}
T\colon& [0,1]\longrightarrow[0,1]\nonumber\\
&\phantomarrow{[0,1]}{x}\left\{\begin{array}{cc}
g(x) & 0\leq x<\frac{1}{6} \left(3-\sqrt{3}\right) \\
 \frac{x+\frac{1}{6} \left(\sqrt{3}-3\right)}{\frac{1}{6} \left(\sqrt{3}-3\right)+\frac{1}{6} \left(3+\sqrt{3}\right)} & \frac{1}{6} \left(3-\sqrt{3}\right)\leq x<\frac{1}{6} \left(3+\sqrt{3}\right) \\
 g(x) & \frac{1}{6} \left(3+\sqrt{3}\right)\leq x<1 \\
\end{array}\right.,
\label{eq:Diff-Cantor-compatible}
\end{align}
which is compatible with the structure of the Cantor set $\Lambda$ since its left and right branches coincide with the map $g$ that generated $\Lambda$, just as $F$ was compatible with $G$ in Section~\ref{subsec:EVL-general-Cantor-sets}. The second is the linear system $T:[0,1]\to[0,1]$, where $T(x)=5x \mod 1$, which, \emph{a priori}, has no reason to be compatible with the geometric structure of $\Lambda$. Both these systems are full branched Markov maps, which means that have decay of correlations against $L^1$ observables.  We note that if we adapt the the arguments presented in Sections~\ref{se:3x_mod1} and \ref{subsec:D-D'}, one could check that conditions $\D(u_n,w_n)$ and $\D'(u_n,w_n)$ hold for these systems and the observable $\varphi$ defined in \eqref{eq:observable-diff-Cantor}. Hence, these examples fit the theory and we expect the existence of an EVL, but the analytical computation of the EI is much more complicated and cannot be carried as for the ternary Cantor set, in Sections~\ref{se:3x_mod1} and \ref{subsec:dimension-EI}.

As in the usual ternary Cantor set and the linear dynamics, the EI easily detects the compatibility between the dynamics and the fractal structure of $\Lambda$. In the first case, where $T$ is given by \eqref{eq:Diff-Cantor-compatible}, the numerical simulations reveal an EI approximately equal to $0.61$, which is consistent with the expected connection between $g$ and $T$, while in the second case, where $T(x)=5x\mod1$, we obtain an EI equal to 1 (see Figure~\ref{fig:ddcs}).

\begin{figure}
\includegraphics[width=7cm]{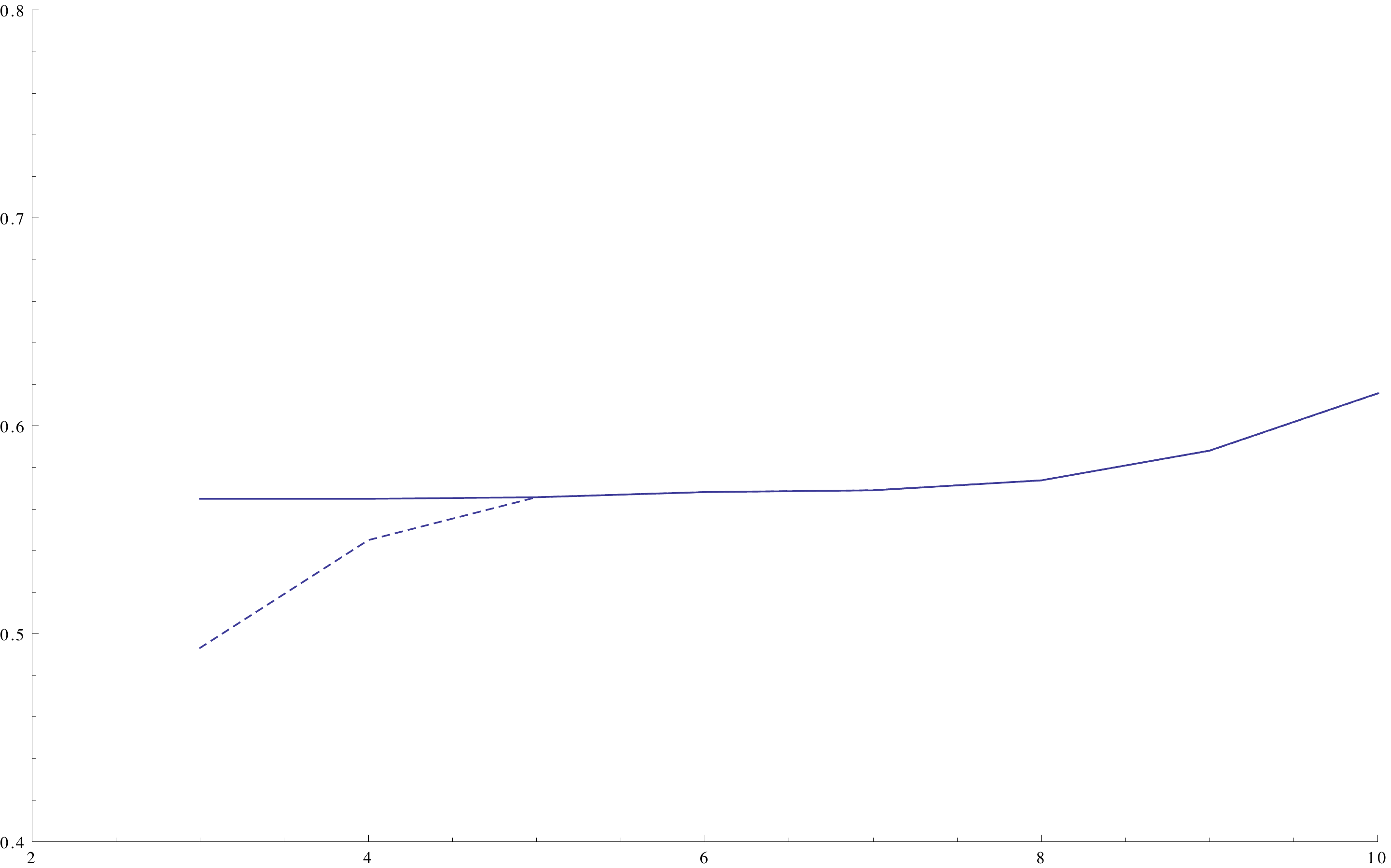} \includegraphics[width=7cm]{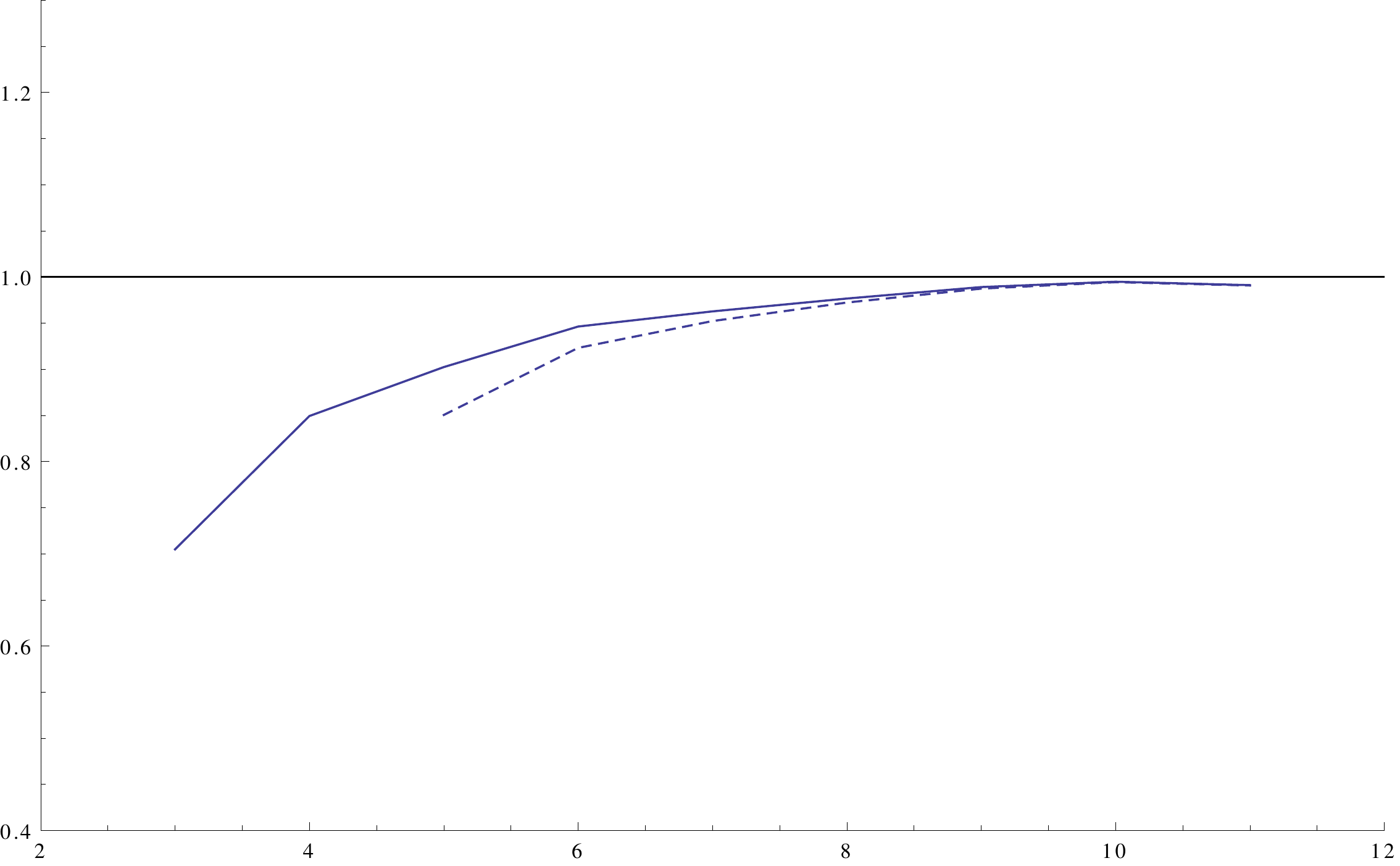} 
\caption{On the $y$-axis, mean values of $\hat \theta_n (u,q)$ for each $5\leq u\leq 20$ of the $x$-axis, with $n=50.000$ and $\ell=500$. The full line corresponds to $q=1$, the dashed line to $q=5$ and the dotted line to $q=10$. On the left $T$ is given by \eqref{eq:Diff-Cantor-compatible} and on the right $T$ is given by $ T(x)=5x\mod1$. }
\label{fig:ddcs} 
\end{figure}

\bibliographystyle{abbrv}

\bibliography{Cantor-biblio.bib}
\end{document}